\documentclass[letterpaper,11pt]{article}

\usepackage{endnotes}
\let\footnote=\endnote

\usepackage{epsfig}
\usepackage{calc}
\usepackage{amstext}
\usepackage{amsmath}
\usepackage{multicol}
\usepackage{amsfonts}
\usepackage{amssymb}
\usepackage{mathrsfs}
\usepackage{bm}
\usepackage{xcolor}

\newcommand{\AR}[2]{\left[\begin{array}{#1}#2\end{array}\right]}
\newcommand{\trace}{\mathrm{trace}}

\newcommand{\rank}{\mathrm{rank}}
\newcommand{\T}{\mathrm{T}}
\newcommand{\SDP}{\mathrm{SDP}}
\newcommand{\OPT}{\mathrm{OPT}}

\newtheorem{theorem}{Theorem}

\newtheorem{proposition}[theorem]{Proposition}

\newtheorem{proof}{Proof}
\usepackage{epsfig}
\usepackage{amstext}
\usepackage{amsfonts}
\usepackage{amssymb}
\usepackage{mathrsfs}
\usepackage{bm}
\usepackage{cleveref}

\title{
	Robust affine control of linear stochastic systems}
\author{
	Georgios Kotsalis, Guanghui Lan
	\thanks{H. Milton Stewart School of Industrial \& Systems Engineering, Georgia Institute of Technology, Atlanta, GA 30332.
		(email: {\tt gkotsalis3@gatech.edu,   \tt george.lan@isye.gatech.edu}).}
}

\begin{document}

\maketitle

\begin{abstract}
In this work we provide a computationally tractable procedure for designing affine control policies, applied to  constrained, discrete-time, partially observable, linear systems subject to set bounded disturbances, stochastic noise and potentially Markovian switching over a finite horizon. 

We investigate the situation  when performance specifications are expressed 
via averaged quadratic inequalities on the random state-control trajectory. Our methodology also applies  to steering the density of the state-control trajectory under set bounded uncertainty, a problem that falls in the realm of robust covariance control.

Our developments are based on expanding the notion of affine policies that are functions of  
the so-called ``purified outputs'',  to the class of Markov jump linear systems. This re-parametrization of the set of policies, induces a bi-affine structure in the state and control variables that can further be  exploited via robust optimization techniques, with the approximate inhomogeneous $S$-lemma being the cornerstone. Tractability is understood in the sense that for each type of performance specification considered, an explicit convex program for selecting the parameters specifying the control policy is provided. Furthermore, our convex relaxation is tight within a factor associated to the $S$-lemma. 

Our contributions to the 
existing literature on the subject of robust constrained control lies in the fact that 
we are addressing a wider class of systems than the ones already studied, by including Markovian switching, and the consideration of quadratic inequalities rather than just linear ones. Our work expands on the previous investigations on finite horizon covariance control by addressing the robustness issue and the possibility that the full state may not be available,  therefore enabling the steering of the state-control trajectory density in the presence of disturbances under partial observation.

 \vspace{.1in}
 
 
 \vspace{.07in}
 
 
\end{abstract}

\vspace{0.1cm}

\setcounter{equation}{0}

\section{Introduction}

We  consider a discrete time partially observable linear system over a finite horizon, 
affected by set bounded disturbances, stochastic noise and possibly Markovian switching. 
Performance specifications are expressed via averaged quadratic inequalities on the random state-control trajectory, 
quadratic inequalities on its mean and semidefinite constraints on its covariance matrix.
The task of the decision maker is to design a non-anticipative control policy so that the constraints are fulfilled 
for all realizations of the  set bounded uncertainty once the random effects of the noise and switching have been averaged out.  Our contribution lies in providing a computationally tractable procedure for designing affine policies that guarantee the robust satisfaction of the constraints.

While the particular setup is novel this problem falls in the category of  robust constrained control and multistage decision making under uncertainty pertinent to the fields of control theory and dynamic optimization respectively. 
Problems of this nature are ubiquitous in science, engineering, and 
finance, leading to the development of a long standing list of solution methodologies, including 
dynamic programming (DP) and the associated variational-formulation-based techniques
\cite{fleming}, \cite{bertsekas1},
approximate dynamic programming \cite{sutton}, \cite{bertsekas2},
stochastic programing \cite{birge}, \cite{shapiro},
set-theoretic control methods \cite{blanchini}, \cite{kurzhanski}, model predictive control \cite{rawlings}, \cite{maj},
$l_1$ methods \cite{dahleh}, \cite{elia} and robust control 
\cite{zhou}, \cite{basar}, \cite{dullerud}, as well as, 
robust optimization \cite{bental1}.

The most general system class considered in this paper is that of Markov jump linear systems. They allow
one to capture the changes that the underlying parameters of a system may undergo as it switches among different 
operating conditions. 
Fundamental contributions on the stochastic jump linear continuous time control problem 
can be found in \cite{krasovskii}, \cite{sworder} and \cite{wohnham}.
Subsequently, various analysis and synthesis results applicable
to linear time-invariant systems have been extended to
 Markov jump linear systems. A comprehensive review of
this material can be found in \cite{costa}, \cite{morozan} and the references therein.
Given their modeling flexibility Markov jump linear systems have been employed extensively in applications ranging from robotics to communications and networked control \cite{seiler}, \cite{hespanha}, 
and apart from engineering in fields such as operations research and econometrics where we refer indicatively to \cite{kwon}, \cite{tu} and \cite{hamilton} respectively. 
In the latter two fields the more common term is regime-switchinig models. 

The uncertainty sets we consider are expressed as as the intersection of
finitely many ellipsoids and elliptic cylinders. It is a subclass of 
ellitopes the latter ones introduced in the context of nonparametric statistics in \cite{jud_17}.

The constraints that we impose on the mean and covariance matrix of the state-control trajectory naturally connect to the problem of steering  
its density  in the presence of set bounded uncertainty. 
This particular problem that we study falls in the category of robust covariance control under partial observation. 
The covariance steering problem over a finite horizon under full state feedback in the absence of set bounded disturbances has been investigated in \cite{chen_1}, \cite{chen_2}, \cite{chen_3},
\cite{okamoto}, \cite{bakolas}.  In an infinite horizon setting  the literature goes back to the original papers on the steady state state-covariance 
assignment problem \cite{hotz1}, \cite{hotz2}. Subsequent work on the infinite horizon covariance assignment problem includes \cite{xu}, \cite{yasuda}, \cite{grigoriadis},
see also \cite{skelton} and the references therein. 

Our work falls also into the realm of affinely adjustable robust counterparts (AARC) \cite{bental3}, \cite{bental1},
for robust optimization problems, where optimization variables correspond to decisions that are made sequentially
as the actual values of the uncertain data become  available. The literature on the topic of robust optimization in static and  multistage settings is extensive; see, for example, \cite{bental4}, \cite{bental5}, \cite{bertsimas1}, \cite{sim1} and \cite{brown1}.

We consider  \cite{bental2}  as the closest work to the methods described in this paper. 
In \cite{bental2} the authors investigated
the robust constrained multistage decision problem for  deterministic linear dynamic models 
subject to linear constraints on the state-control trajectory. 
In our case the stochasticity induced by the Markovian-switching dynamics and the external noise dictates the use of quadratic constraints, instead of linear ones, to meaningfully confine not only the mean of the state-control trajectory,
but also keep its covariance matrix within a desired range. The price one pays for considering the wider class of constraints is that the resulting convex formulation for the affine policy design is a semidefinite program, as opposed to the corresponding linear programming formulation encountered in \cite{bental2}. 

We restrict ourselves to affine policies and in particular a specific parameterization thereof, affine policies based on ``purified outputs'' for the Markovian-switching models under consideration; see also \cite{bental2} for the corresponding concept in the deterministic case.  The purified output parametrization can be interpreted as a finite horizon analog of the Youla parametrization \cite{youla1}, \cite{youla2} and the related $Q$-design procedure see for instance \cite{barratt}, \cite{vid}.
The main advantage of this restriction in affine purified output-based policies 
lies in the resulting efficient computation of the policy parameters by an explicit convex optimization problem. This appealing feature comes at the cost of potential conservatism.

In the scalar case, under the assumption of deterministic dynamics, the optimality of affine policies for a class of robust multistage
optimization problems subject to constraints on the control variables was proven in 
\cite{bertsimas2}.  The degree of suboptimality of affine policies was investigated computationally in \cite{bertsimas3} and \cite{kuhn}. The former reference employs 
the sum of squares relaxation, while the latter efficiently computes the suboptimality gap by invoking linear decision rules in the dual formulation as well. 
Instances  where affine policies are provably 
optimal also include the linear quadratic Gaussian problem (LQG) \cite{kwakernaak}, \cite{andreson} 
as well as its risk sensitive version with an exponential cost criterion (LEQG) \cite{whittle2}, both problems admitting
a dynamic programming solution.  

Further applications of affine policies can found in stochastic programming settings \cite{sun}, \cite{nemirovski}
and in the area of robust model predictive and receding horizon control \cite{loefberg}, \cite{goulart1}, \cite{goulart2},
\cite{skaf}.  From a historical perspective  some of the earliest accounts on the use of affine policies in stochastic programming 
include \cite{charnes}, \cite{gartska}. The constrained control problem in he presence of disturbances has a long history as well. 
Fundamental contributions include the work on minimax controller design \cite{salm}, \cite{wits} and the minimax reachability 
problem \cite{glover}, \cite{ber2}, \cite{ber3}, \cite{ber4}. 

In retrospect the computational tractability of our approach is owed to the particular choice of uncertainty set in conjuction to the restriction to affine policies in purified outputs 
and the subsequent use of the approximate inhomogeneous $S$-lemma. The $S$-lemma plays a prominent role both in robust optimization and control theory, where it is typically referred to as $S$-procedure \cite{yakubovich1}. The extension to infinite dimensional settings \cite{meg}, \cite{yakubovich2} led to many applications in robust control problems, see for instance \cite{fer}, \cite{savkin} and the references therein.

The paper is organized as follows.  Section 2 contains some common notation that will be used in the subsequent discussion. 
The formal problem statement is given in
Section 3 which concludes with a statement of our contributions. The main theorems of the paper describing our tractable approach are in Section 4. We provide a numerical 
illustration of our methodology in Section 5, and we conclude with a summary and thoughts for future directions in Section 6.

\section{Notation}
The set of positive integers is denoted by $ \mathbb{Z}_{+} $, the set of non-negative integers by
$\mathbb{N} $,
the set of real numbers by $ \mathbb{R}$. For $m,n \in \mathbb{N}$, $m< n $,
$ \mathbb{Z}_{[m,n]} = \{m,m+1, \hdots, n \}$.
For $n,m \in \mathbb{Z}_{+}$ let
$ \mathbb{R}^n$ denote the Euclidean $n$-space and  $ \mathbb{R}^{n \times m}$ 
the set of $n$-by-$m$ real matrices.  The set of symmetric matrices in 
$\mathbb{R}^{n \times n}$ is $\mathbb{S}^n$, the positive semi-definite, (definite) cone is $ \mathbb{S}^n_+, (\mathbb{S}_{++}^n)$ and the relation
$A \succeq ( \succ) B$ stands for $ A - B \in   \mathbb{S}^n_+, ( \mathbb{S}^n_{++})$.
The transpose of the column vector $ x \in \mathbb{R}^n$ is denoted by $ x^\T$ and for $i \in \mathbb{Z}_{+}$, $ [x]_i$ refers to the $i$-th entry of $ x$. Similarly for $ i, j \in \mathbb{Z}_{+}$,
$[ A]_{i,j}$ refers to the $(i,j)$-th entry of the matrix $ A \in \mathbb{R}^{n \times m}$. 
For $ x, y  \in \mathbb{R}^n$, $ \langle x, y \rangle$ denotes the inner product of the two vectors. 
A sum (product) of $n  \times n$ matrices over an empty set of indices is the 
$n \times n$ zero (identity) matrix.

Random vectors are denoted in boldface and are defined on an abstract probability space $( \Omega, \mathcal{F}, \mathbb{P})$,  $ \mathbb{E}[\cdot]$ denotes the expectation operator. 
Given $t  \in \mathbb{Z}_+$ and a  discrete time random process $\mathbf{x}_0, \mathbf{x}_1, \hdots$, 
the notation $\mathbf{x}_0^t$ stands for the history of the process from time $0$ to time $t$, i.e. $ \mathbf{x}_0^t = \{\mathbf{x}_0, \mathbf{x}_1, \hdots, \mathbf{x}_t \}$. A realization, sample path,  of $\mathbf{x}_0^t$ is denoted by $x_0^t = \{x_0, x_1, \hdots, x_t \}$.

\section{Problem Statement}

\subsection{System}

Fix some positive integers  $n_x, n_u, n_{d},  n_{e}, n_y,  N, m $.
We will consider a partially observable, 
discrete-time, Markov jump linear system on a finite time horizon $N$, referred to as $\mathcal{S}$, with state space description 
 \begin{eqnarray}
\label{system}
\nonumber 
\mathbf{x}_0 & = & z ~ + ~ \mathbf{s}_0, \\
\nonumber 
\mathbf{x}_{t+1} & = &  A_t[\bm{\theta}_t]  ~ \mathbf{x}_t ~ + ~ B_t[\bm{\theta}_t] ~  \mathbf{u}_t ~ + ~ \bm{\sigma}_t \\
\mathbf{y}_t  & = & C_t[\bm{\theta}_t ] ~  \mathbf{x}_t ~ + ~ 
\bm{\rho}_t,~~~ t \in \mathbb{Z}_{[0,N-1]}.
\end{eqnarray}
The switching signal $ \bm{\theta}_{t}$ is modeled as a Markov chain, 
taking values in $ \mathbb{O} = \{1, 2, \hdots, m\}$ with initial distribution 
$\pi$ and probability transition matrix $P$. The state of the system is hybrid
composed of $ ( \mathbf{x}_t, \bm{\theta}_t) \in \mathbb{R}^{n_x} \times \mathbb{O}$ and the control vector is $ \mathbf{u}_t \in \mathbb{R}^{n_u}$.
We assume that the underlying switching process is observable and as such the
output of the system is hybrid, composed of  
$ ( \mathbf{y}_t, \bm{\theta}_t) \in \mathbb{R}^{n_y} \times \mathbb{O}$.
For simplicity, we will 
incur an abuse of notation 
by referring solely to $ \mathbf{x}_t$
and $ \mathbf{y}_t$ as the state and output of the system respectively and to 
$ \bm{\theta}_t$ as the switching signal. 
The initial state $\mathbf{x}_0$ exhibits a deterministic and a stochastic component.
In particular for a given  $\Sigma_0 \in \mathbb{S}^{n_x}_{++}$,
$$
\mathbf{x}_0 = z + \mathbf{s}_0, ~~~ \text{where}~~ \mathbf{s}_0 \sim  \mathcal{N}(0, \Sigma_0).
$$  
The disturbances $ \bm{\sigma}_t, \bm{\rho}_t$ are induced by both set bounded uncertainty and stochastic noise,
\begin{eqnarray}
\label{disturbance model}
\nonumber
\bm{\sigma}_t & = & B_t^{(d)}[\bm{\theta}_t]  ~   d_t ~+~  
B_t^{(s)}[\bm{\theta}_t ]  ~ \mathbf{e}_t, \\
\bm{\rho}_t & = & D_t^{(d)}[\bm{\theta}_t ] ~  d_t ~  + ~ D_t^{(s)}[\bm{\theta}_t ]   ~  \mathbf{e}_t, ~~~t \in \mathbb{Z}_{[0,N-1]}.
\end{eqnarray}
The set bounded disturbance is $d_t \in 
\mathbb{R}^{n_{d}}$ and the stochastic disturbance is $ \mathbf{e}_t \in
\mathbb{R}^{n_{e}}$ respectively. 
It is assumed that $\mathbf{s}_0, \mathbf{e}_0^{N-1}$ are jointly independent
and for each $t \in \mathbb{Z}_{[0,N-1]} $, 
$$
\mathbf{e}_t \sim \mathcal{N}(0,I).
$$
We refer to the random variables  $ \bm{\theta}_0^{N-1}, \mathbf{s}_0,  
\mathbf{e}_0^{N-1}$ as the basic random variables of the system description,
since all other random variables are constructed from them. 
Furthermore $ \bm{\theta}_0^{N-1}$ is independent of $\mathbf{s}_0,  
\mathbf{e}_0^{N-1} $.  The switching signal  controls the selection 
of the system matrices 
$$A_t[\bm{\theta}_t], B_t[\bm{\theta}_t], 
B_t^{(d)}[\bm{\theta}_t], B_t^{(s)}[\bm{\theta}_t], 
C_t[\bm{\theta}_t], D_t^{(d)}[\bm{\theta}_t],  D_t^{(s)}[\bm{\theta}_t],$$ 
which are assumed to be of conformable dimensions and depend in a manner known in advance on its current realization $\theta_t$.
We will  write 
$$
\bm{\epsilon}^{\T} =\AR{ccccc}{\mathbf{s}_0^\T, &  \mathbf{e}_0^\T, & \mathbf{e}_1^\T,& \hdots, & \mathbf{e}_{N-1}^\T} 
\in  \mathbb{R}^{n_{\epsilon}},
$$
where
$n_{\epsilon} =  n_x + N  n_e$.
The signals $z, d_0^{N-1}$ introduce set bounded uncertainty into the system $\mathcal{S}$. The uncertainty vector is $\zeta$ where
$$
\zeta^\T= \AR{ccccc}{z^\T, &  d_0^\T, & d_1^\T,& \hdots, & d_{N-1}^\T} 
\in  \mathscr{D}^N \subset \mathbb{R}^{n_{\zeta}}, 
$$
$n_{\zeta} =  n_x + N  n_d$, and $\mathscr{D}^N $ is the uncertainty set.

\subsection{Uncertainty set description}

Our choice for $\mathscr{D}^N $ is dictated by the need of upper-bounding 
the maximum of a quadratic form over the uncertainty set. We consider 
an intersection of ellipsoids parametrized by positive semidefinite matrices 
$$ 
\{Q_i \in \mathbb{S}_+^{n_{\zeta}}  \}_{i \in  \mathbb{Z}_{[1,s]}},~~ 
\text{such that}~~ \sum_{i=1}^s Q_i \succ 0.
$$
The uncertainty set 
$\mathscr{D}^N $
is expressed as 
\begin{eqnarray}
\label{ellitop}
\mathscr{D}^N  = \{ \zeta \in \mathbb{R}^{n_{\zeta}} ~|~   \langle Q_i \zeta, \zeta \rangle \leq 1, ~ 1 \leq k \leq s \}.
\end{eqnarray}
The parameters $  \{  Q_i,  \}_{i \in \mathbb{Z}_{[1,s]}} $ describing  $\mathscr{D}^N$
are assumed to be part of the problem specification. 
The particular choice of uncertainty description is related to the subsequent use of the approximate inhomogeneous $S$-lemma for the purpose of arriving to a computationally tractable control design methodology.

\subsection{Affine control laws in purified outputs}

We will employ causal, deterministic control policies that are affine in the purified outputs of the system.  In order to introduce  the concept of 
purified-output-based (POB) control laws suppose first that the system $\mathcal{S}$ is controlled by a causal (non-anticipating) policy
$$
\mathbf{u}_t = \phi_t( \mathbf{y}_0, \hdots, \mathbf{y}_t), ~~~t  \in \mathbb{Z}_{[0,N-1]},
$$
where for time instant, $\mathbf{u}_t$ is adapted to the filtration
associated to the underlying probability space. 
We consider now the evolution of a disturbance free system starting from rest, with 
state space description
\begin{eqnarray}
	\label{model}
	\nonumber
	\hat{\mathbf{x}}_0 & = & 0, \\
	\nonumber
	\hat{\mathbf{x}}_{t+1} & = &  A_t[\bm{\theta}_t] ~ \hat{\mathbf{x}}_t ~ + ~ B_t[\bm{\theta}_t] ~ \mathbf{u}_t, \\
	\hat{\mathbf{y}}_t  & = & C_t[\bm{\theta}_t] ~ \hat{\mathbf{x}}_t,  ~  t \in  \mathbb{Z}_{[0,N-1]},
\end{eqnarray}
where $ \hat{\mathbf{x}}_t \in \mathbb{R}^{n_\mathbf{x}}$, $\hat{\mathbf{y}}_t \in \mathbb{R}^{n_y}$.
The dynamics in \eqref{model} can be run in ``on-line'' fashion given that 
the decision maker observes not only the realizations of the output and the switching signals,  
but also knows the already computed control values.  
The purified outputs are the difference of the outputs of the actual system \eqref{system} and the outputs of the auxiliary 
system \eqref{model}, i.e. 
\begin{equation}
	\label{pure_out}
	\mathbf{v}_t = \mathbf{y}_t - \hat{\mathbf{y}}_t , ~  t \in  \mathbb{Z}_{[0,N-1]}.
\end{equation}
At time instant $t \in \mathbb{Z}_{[0,N-1]}$ the realization of the purified output 
$v_0^t$ is available when the decision on selecting the control value $u_t$ is to be made. 
The term ``purified'' stems from the fact that the outputs \eqref{pure_out} 
do not depend on the particular choice of control policy.  This circumstance is verified by introducing the signal 
$$
\bm{\delta}_t = \mathbf{\mathbf{x}}_t - \hat{\mathbf{x}}_t,~~~t \in  \mathbb{Z}_{[0,N-1]},
$$
and observing that the purified outputs can also be obtained 
by combining \eqref{system}, \eqref{model}, \eqref{pure_out} and the definition of $\bm{\delta}_t$ as 
the outputs of the system with state space description
\begin{eqnarray}
\label{model_pure}
\nonumber
\bm{\delta}_0 & = & \mathbf{x}_0,\\
\nonumber
\bm{\delta}_{t+1} & = & A_t[\bm{\theta}_t] ~ \bm{\delta}_t ~ + ~ B_t^{(d)}[\bm{\theta}_t] ~ d_t ~ + ~ B_t^{(s)}[\bm{\theta}_t] ~ \mathbf{e}_t,   \\
\mathbf{v}_t  & = & C_t[\bm{\theta}_t] ~ \bm{\delta}_t  ~ + ~ D_t^{(d)}[\bm{\theta}_t] ~ d_t ~ + ~ D_t^{(s)}[\bm{\theta}_t] ~ \mathbf{e}_t,  
\end{eqnarray}
$t=0,1, \hdots, N-1$. For $T \in \mathbb{Z}_+$, let 
\begin{eqnarray*}
\bm{\theta}_{[t,T]} & = &
\AR{c}{\theta_{\max{\{ 0,t-T \}}} \\ \theta_{\max{\{ 0,t-T+1 \}}} \\  \vdots  	\\ \theta_t}  \in 
\mathbb{O}^{\min{\{T+1,t+1\}}}, \\
\bm{\theta} &=&  \bm{\theta}_{[N-1,N-1]} \in \mathbb{O}^N.
\end{eqnarray*}
An affine POB policy is of the form
\begin{equation}
\label{control_pure_out}
\mathbf{u}_t = h_{t}[\bm{\theta}_{[t,T]}] + \sum_{i=0}^t
H_i^t[\bm{\theta}_{[t,T]}]~ \mathbf{v}_i,~~
~t \in  \mathbb{Z}_{[0,N-1]},
\end{equation}
where
\begin{equation}
\label{parameters_control}
h_{t} :  \mathbb{O}^{\min{\{T+1,t+1\}}}  \rightarrow \mathbb{R}^{n_u},~ H_j^t :    \mathbb{O}^{\min{\{T+1,t+1\}}} \rightarrow \mathbb{R}^{n_u \times n_x},
\end{equation}
$0 \leq j \leq t, ~~ t \in \mathbb{Z}_{[0,N-1]}$. 

In the above description $T$ is the length of the ``switching memory''  utilized by the affine policy under consideration. 
The parameters of the POB affine policy are the collection of maps
$$ \{ h_{t}, H_j^t \}_{0 \leq j \leq t, ~ t \in \mathbb{Z}_{[0,N-1]}} \in \mathcal{X},
$$ 
where
$ \mathcal{X}$ denotes the set of all possible POB affine control laws 
applicable to the system $\mathcal{S}$.

Inspecting the affine policy \eqref{control_pure_out} we observe that at a given time instant $t$ the number of decision variables 
associated with the term $h_{t}[\bm{\theta}_{[t,T]}]$ 
is 
$$
\begin{cases}
m^{T+1}   n_u ~~~ \text{if}~~~ t \geq T \\
m^{t+1}    n_u ~~~ \text{if}~~~ t <  T.
\end{cases}
$$
At the same time instant $t$ the number of decision variables 
associated with the terms 
$$ H_0^t[\bm{\theta}_{[t,T]}], H_1^t[\bm{\theta}_{[t,T]}], 
\hdots, H_t^t[\bm{\theta}_{[t,T]}]$$
is 
$$
\begin{cases}
(t+1)   m^{T+1}  n_u  n_y ~~~ \text{if}~~~ t \geq T \\
(t+1)    m^{t+1}     n_u  n_y~~~ \text{if}~~~ t <  T.
\end{cases}
$$
Thus the  number of decision variables associated with the $t$-th level is
$
m^{T+1}    n_u  ( (t+1)   n_y + 1)
$
and the total number of decision variables associated to the policy parameters is 
\begin{eqnarray}
\nonumber
\label{dimension}
\dim(N,T) & = & \sum_{i=0}^T m^{i+1} n_u ( (i+1)   n_y + 1) \\~ &+& ~  \sum_{i=T+1}^{N-1} m^{T+1}   n_u 
( (i+1)    n_y + 1).
\end{eqnarray}
In the following we will identify the policy parameters $ \{ h_{t}, H_j^t \}_{0 \leq j \leq t, t \in \mathbb{Z}_{[0,N-1]}} $ with an element 
$ \chi \in \mathbb{R}^{\dim(N,T)}$ that contains their values for all possible realizations of the switching signal $ \bm{\theta} \in 
\mathbb{O}^N$. This identification requires some ordering on the finite set $ \mathbb{O}^N$, 
take for instance the first lexicographic ordering.  
In a slight abuse of language we will associate  the policy parameters interchangeably to
$ \chi$
or 
$\{ h_{t}, H_j^t \}_{0 \leq j \leq t, t \in \mathbb{Z}_{[0,N-1]}}$,
while realizing that the former is a vector and  the latter 
is a collection of maps, given that they describe the same entity. 

In this work we will  consider the control design problem strictly in the context of POB affine control laws, the sole motivation being the resulting computational tractability as it will be demonstrated in the convex design section. The closed loop system consists of 
the  open loop dynamics  \eqref{system} coupled with the model based dynamics  \eqref{model}, the purified outputs \eqref{pure_out} and the POB affine policies \eqref{control_pure_out}.

\subsection{Performance specifications}

The linear system $\mathcal{S}$ is subject to both stochastic noise as well 
as set bounded disturbances. Performance specifications will be expressed via constraints on the random state-control trajectory and its first two moments. 
At any instant of time, 
both the control vector $\mathbf{u}_t$ and state $\mathbf{\mathbf{x}}_t$ of the system are Gaussian vectors and with slight abuse of notation we introduce stacked versions of them, denoted by  $\mathbf{x}, \mathbf{u}$,
\begin{eqnarray}
\label{state-control-trajectory}
\mathbf{x} &= & \AR{c}{  \mathbf{x}_1 \\ \mathbf{x}_2 \\  \vdots \\ \mathbf{x}_{N}},~
\mathbf{u} =  \AR{c}{  \mathbf{u}_0, \\  \mathbf{u}_1 \\ \vdots \\ \mathbf{u}_{N-1}},~
\mathbf{w} = \AR{c}{\mathbf{x} \\  \mathbf{u} },
\end{eqnarray}
where $\mathbf{x} \in \mathbb{R}^{n_x N}$,  $\mathbf{U} \in \mathbb{R}^{n_u N}$, $\mathbf{w} \in \mathbb{R}^{n_w}$, $n_w = n_x  N +  n_u  N$. 
Accordingly we write for the respective mean and covariance matrix
$$
\mu_{\mathbf{w}} = \mathbb{E}[\mathbf{w}],~~\Sigma_{\mathbf{w}} = \mathbb{E}[(\mathbf{w} - \mu_{\mathbf{w}}) (\mathbf{w}- \mu_{\mathbf{w}})^{\T} ].
$$
The associated expectations are to be computed with respect to 
the joint distribution of the basic random variables  $ \bm{\theta}_0^{N-1}, \mathbf{s}_0,  
\mathbf{e}_0^{N-1}$.  We will also consider the mean and covariance matrix
of the state and control vector at specific instances in time, these
quantities are denoted as
$$
\mu_{\mathbf{x},t} = \mathbb{E}[\mathbf{x}_t],~
\Sigma_{\mathbf{x},t} = \mathbb{E}[(\mathbf{x}_t- \mu_{\mathbf{x},t}) (\mathbf{x}_t- \mu_{\mathbf{x},t})^{T}],~ 
t \in \mathbb{Z}_{[0,N]},
$$
and 
$$
\mu_{\mathbf{u},t} = \mathbb{E}[\mathbf{u}_t],~
\Sigma_{\mathbf{u},t} = \mathbb{E}[(\mathbf{u}_t- \mu_{\mathbf{u},t}) (\mathbf{u}_t- \mu_{\mathbf{u},t})^{T}],
~
t \in \mathbb{Z}_{[0,N-1]},
$$
respectively. The specifications imposed directly on the state-control trajectory are expressed via averaged quadratic inequalities. 
The constraints on its first two moments, relating to the robust density steering problem, consist of quadratic inequalities on its mean $\mu_\mathbf{w} $ and prescribing upper bounds, in the matrix sense, 
on weighted versions of the covariance matrix $ \Sigma_{\mathbf{w}}$.

The common theme between the constraints 
under consideration is the resulting explicit convex programming formulation for designing POB affine policies such that the closed loop system satisfies them in a robust manner.

\subsubsection{Averaged quadratic inequalities in $\mathbf{w}$}

Let $ r \in \mathbb{Z}_+$. We will consider averaged quadratic inequality constraints on the state-control trajectory  expressed as
\begin{equation}
\label{specification_averaged_quadratic}
\forall   \zeta \in \mathscr{D}^N, ~
\mathbb{E}[\langle \mathcal{A}_i ~  \mathbf{w},   \mathbf{w} \rangle + 2 ~ \langle \beta_i ,  \mathbf{w} \rangle ] ~ \leq~ \gamma_i, 
\end{equation}
where $i \in \mathbb{Z}_{[1,r]}$ and 
\begin{equation}
\label{parameters_averaged_quadratic}
\mathcal{A}_i \in \mathbb{S}_+^{n_w},~ \beta_i \in \mathbb{R}^{n_w }, ~\gamma_i \in \mathbb{R}.
\end{equation}
The parameters $ \{ \mathcal{A}_i, \beta_i, \gamma_i \}_{i \in
	\mathbb{Z}_r}$ defining the quadratic inequalities are assumed to be part of the problem specification.  The constraints in \eqref{specification_averaged_quadratic} once averaged out 
are to be interpreted as ``hard'' constraints in the 
uncertainty vector $ \zeta \in \mathscr{D}^N$.

Consider the left hand side of any of the inequalities that appear in \eqref{specification_averaged_quadratic}, say $ i \in \mathbb{Z}_{[1,r]}$, 
$$ \mathbb{E}[\langle \mathcal{A}_i ~  \mathbf{w},   \mathbf{w} \rangle + 2 ~ \langle \beta_i ,  \mathbf{w} \rangle ],$$  and note that we do not require for it 
to exhibit time separability, i.e. stage additivity, in its functional form.
As such our methodology for designing causal feedback laws is applicable to situations that may not be amenable to dynamic programming type arguments.

A particular scenario where such constraints are relevant consists of the decision maker having some desired targets for the state and control vector as well as allowable deviation levels to be fulfilled 
in expectation for all realizations of $ \zeta \in \mathscr{D}^N$,
for some subsets of the finite time horizon $   T_x \subset \mathbb{Z}_{[1,N]}$ and $   T_u \subset \mathbb{Z}_{[0,N-1]}$
respectively, i.e. for every realization of the uncertainty vector $ \zeta \in \mathscr{D}^N$,
\begin{eqnarray*}
	\mathbb{E}[\| \mathbf{x}_t - \hat{x}_t \|^2]  \leq \hat{\gamma}_t, ~  t \in T_x, ~~
	\mathbb{E}[\| \mathbf{u}_t - \tilde{u}_t \|^2]  \leq \tilde{\gamma}_t, ~  t \in T_u.
\end{eqnarray*}
The above formulation is one of a robust feasibility problem,  where 
$\{\hat{\gamma}_t\}_{t \in T_x}$, $\{ \tilde{\gamma}_t\}_{t \in T_u}$ are a priori fixed. However, 
our framework can be readily extended by the minimization of a convex function of the deviation levels $\{\hat{\gamma}_t\}_{
	t \in T_x}$ and 
$\{ \tilde{\gamma}_t\}_{t \in T_u}$ and as such streamlining the selection
of the latter,  while preserving tractability.

We also mention that within our scope fall problems where the
worst case value over $\zeta  \in \mathscr{D}^N$ of the 
expectation of a convex quadratic objective in the random state-control trajectory is to be  minimized over $ \chi \in \mathcal{X}$, i.e.
$$
\min_{ \chi \in \mathcal{X}} \max_{\zeta  \in \mathscr{D}^N }
\mathbb{E}[\langle \mathcal{A} ~  \mathbf{w},   \mathbf{w} \rangle + 2 ~ \langle \beta,  \mathbf{w} \rangle ]
$$
subject to further averaged quadratic inequalities that are to be satisfied robustly in the sense of \eqref{specification_averaged_quadratic}.

\subsubsection{Robust steering of densities}

The system $\mathcal{S}$ is subject to set bounded disturbances $ \zeta \in \mathscr{D}^N$ and the question arises whether one can confine the  mean and covariance 
matrix of the state-control trajectory in a certain prespecified set. 
To this end 
we will prescribe constraints to the mean of the state-control trajectory via quadratic inequalities. Let $ \hat{r} \in \mathbb{Z}_+$ and consider the constraints 
\begin{equation}
\label{specification_mean_quadratic}
\forall   \zeta \in \mathscr{D}^N, ~ \langle \hat{\mathcal{A}}_i ~  \mu_\mathbf{w},   \mu_\mathbf{w} \rangle + 2 ~ \langle \hat{\beta}_i ,  \mu_\mathbf{w} \rangle  ~ \leq~ \hat{\gamma}_i, 
\end{equation}
where $ i \in \mathbb{Z}_{[1,\hat{r}]}$ and
\begin{equation}
\label{parameters_mean_quadratic}
\hat{\mathcal{A}}_i \in \mathbb{S}_+^{n_w},~ \hat{\beta}_i \in \mathbb{R}^{n_w }, ~\hat{\gamma}_i \in \mathbb{R}.
\end{equation}
The parameters $ \{ \hat{\mathcal{A}}_i,  \hat{\beta}_i, \hat{\gamma}_i \}_{i \in \mathbb{Z}_{[1,\hat{r}]}}$ defining the quadratic inequalities are assumed to be part of the problem specification.
By appropriate choice of the parameters the constraints in \eqref{specification_mean_quadratic} allow 
one to express the requirement 
of robustly confining the mean of the state and the control vector
at different instances of time in prespecified ellipsoidal or polyhedral sets, in particular once can express constaints of the form, for all realizations $ \zeta \in \mathscr{D}^N$,
\begin{eqnarray*}
	\| \mu_{\mathbf{x},t} - \hat{\mu}_t \|^2  & \leq &  \hat{\epsilon}_t, ~
	\hat{\gamma}_t^{(-)} \leq \mu_{\mathbf{x},t}  \leq  \hat{\gamma}_t^{(+)},  ~~ t \in T_x, \\
	\| \mu_{\mathbf{u},t} - \tilde{\mu}_t \|^2 & \leq &   \tilde{\epsilon}_t, ~
	\tilde{\gamma}_t^{(-)} \leq \mu_{\mathbf{u},t}  \leq  \tilde{\gamma}_t^{(+)},  ~~t \in T_u. 
\end{eqnarray*}

As for the covariance matrix of the state-control trajectory, let us fix
$ \tilde{r}, n_k \in \mathbb{Z}_+$ and consider constraints of the form
\begin{equation}
\label{specification_covariance}
\forall   \zeta \in \mathscr{D}^N, ~  \mathcal{Q}_i  ~ \Sigma_\mathbf{w}  ~  \mathcal{Q}_i^\T \preceq \tilde{\Sigma}_i, 
\end{equation}
where $i \in \mathbb{Z}_{[1,\tilde{r}]} $
and 
\begin{equation}
\label{parameters_mean_quadratic}
\tilde{\Sigma}_i \in \mathbb{S}^{n_k}_{++},~~ 
\mathcal{Q}_i \in \mathbb{R}^{n_k \times n_w }.
\end{equation}
The parameters $ \{ \tilde{\Sigma}_i, \mathcal{Q}_i \}_{i \in \mathbb{Z}_{[1,\tilde{r}]}}$ appearing in \eqref{specification_covariance}
are assumed to be part of the problem specification.
The constraints in \eqref{specification_covariance} allow the decision maker 
by appropriate choice for instance of the weighting matrices $\{ \mathcal{Q}_i \}_{i \in \mathbb{Z}_{[1,\tilde{r}]}}$
to express requirements of the form, for all $  \zeta \in \mathscr{D}^N$,
\begin{eqnarray*}
	\Sigma_{\mathbf{x},t}  \preceq \hat{\Lambda}_t, ~ t \in T_x,~ 
	\Sigma_{\mathbf{u},t} \preceq \tilde{\Lambda}_t, ~ \forall t \in T_u,
\end{eqnarray*}
where $ \{ \hat{\Lambda}_t \}_{t \in T_x}$,  $ \{ \tilde{\Lambda}_t \}_{t \in T_u}$ are prespecified positive definite matrices.

Finally we point out that our methodology of providing an explicit convex
programming formulation for the control design problem in POB affine policies subject to constraints of the form  \eqref{specification_averaged_quadratic}, \eqref{specification_mean_quadratic}, \eqref{specification_covariance} can be augmented with further
convex constraints on the control parameters that reflect limits on their magnitudes, prescribed sparsity patterns, etc.

\subsection{Contributions}

The main contribution of this paper is the development 
of a computationally tractable procedure for designing affine policies 
in purified outputs for partially observable Markov jump linear systems
affected by external  disturbances and an unknown initial state, while
being subject to quadratic constraints.
Prior work had focused on linear dynamic models exhibiting deterministic evolution  subject to linear constraints on the finite-horizon state-control trajectory.

The advantage of affine policies based on the purified outputs over the traditional affine policy parameterization
where  the bi-affinity of the state-control trajectory
as a function of the policy parameters and the vector of the initial state and disturbances is established. This fact  sets up
theorem 1 that provides an
explicit convex program specifying, when feasible, policy parameters ensuring the validity of the design specifications.   Furthermore,
our convex relaxation is tight within  a factor associated to the approximate inhomogeneous S-lemma.
Theorem 1  serves  as the basis for deriving explicit semi-definite programming formulations for the control
design problem and is further supported by proposition 2 for the case of quadratic inequalities  and proposition 3 for the covariance constraints.  
The tractability of our approach is validated via propositions \ref{lemma1} and \ref{lemma2}.
Namely, in order to avoid exponential exploding in the number of parameters specifying the policies as $N$ grows, 
we restrict the ``switching memory'' of the policies. In that case,  when using policies that employ a truncation of fixed length  of the regime-switching sequence, the
coefficients that enter in our
convex program can be computed in polynomial in $N, M$ arithmetic operations.
This polynomial complexity result is important given that the number of possible sample paths for the underlying regime-switching process grows exponentially in $N$, the horizon under consideration.
The consideration of affine policies based on purified outputs  motivates the question as to their relation to
the classical affine policy parameterization employing the actual outputs of the system.
theorem 6 shows the equivalence of the two types
of policies, when both are allowed to depend on the entire switching sequence.

To the best of our knowledge the proposed tractable approach is the first of its kind for the problems under consideration.

\section{Convex design}

For the given finite horizon $N$ we write the state space recursion of $\mathcal{S}$, \eqref{system}, the model based dynamics \eqref{model_pure}, 
and the POB affine control law \eqref{control_pure_out} in matrix form as 
\begin{eqnarray}
\label{vecx}
\mathbf{x} & = & \underline{\mathbf{A}} ~ \mathbf{x}_0 ~ + ~ \underline{\mathbf{B}} ~ \mathbf{u} + ~ + ~ \underline{\mathbf{B}}^{(d)} ~ \zeta_d ~ + ~ \underline{\mathbf{B}}^{(s)} ~ 
\mathbf{e}, \\
\label{vecv}
\mathbf{v} & = &  \underline{\mathbf{C}}~ \mathbf{x}_0 ~ + ~ \underline{\mathbf{D}}^{(d)} ~ \zeta_d~+~ \underline{\mathbf{D}}^{(s)} ~ \mathbf{e}, \\
\label{vecu}
\mathbf{u} & = &  \underline{\mathbf{h}}~ + ~ \underline{\mathbf{H}} ~  \mathbf{v},
\end{eqnarray}
where we introduced in addition to $\mathbf{x}, \mathbf{u}$, \eqref{state-control-trajectory}, the trajectory 
signals $\zeta_d, \mathbf{e}, \mathbf{v},$
\begin{eqnarray*}
\zeta_d  =  \AR{c}{d_0 \\ d_1 \\ \vdots \\ d_{N-1}}, ~
\mathbf{e}  =   \AR{c}{\mathbf{e}_0 \\ \mathbf{e}_1 \\ \vdots \\ \mathbf{e}_{N-1}},~
\mathbf{v}  =  
 \AR{c}{\mathbf{v}_0 \\ \mathbf{v}_1 \\ \vdots \\ \mathbf{v}_{N-1}}, 
\end{eqnarray*}
while the individual matrices  $\underline{\mathbf{A}}$, $\underline{\mathbf{B}}$, 
$\underline{\mathbf{B}}^{(d)}$, $\underline{\mathbf{B}}^{(s)}$, $\underline{\mathbf{C}}$, 
$\underline{\mathbf{D}}^{(d)}$, $\underline{\mathbf{D}}^{(s)}$, $\underline{\mathbf{h}}$, $\underline{\mathbf{H}}$ are given explicitly in Appendix A.
By combining \eqref{vecx} to \eqref{vecu} we obtain for the state-control trajectory 
\begin{eqnarray}
\label{w_explicit}
 \nonumber
\mathbf{w} & = & \bm{b}~ +~
 \mathbf{w}^{(d)} ~ + ~  \mathbf{w}^{(s)}, \\
 \nonumber
 \\
 \nonumber
 \bm{b} & = & \AR{c}{ \underline{\mathbf{B}} ~ 	\underline{\mathbf{h}} \\ 	\underline{\mathbf{h}}}, ~
 \mathbf{w}^{(d)}   =  \bm{\mathcal{B}}^{(d)}  \zeta, ~
 \mathbf{w}^{(s)} = \bm{\mathcal{B}}^{(s)}  \bm{\epsilon},
 \\
 \nonumber
 \\
\nonumber 
 \bm{\mathcal{B}}^{(d)} & = &  \AR{cc}{\underline{\mathbf{B}} ~ 
	\underline{\mathbf{H}}~ \underline{\mathbf{C}} + \underline{\mathbf{B}} ~ &~  
\underline{\mathbf{B}} ~ 
	\underline{\mathbf{H}}~ \underline{\mathbf{D}}^{(d)} + \underline{\mathbf{B}}^{(d)}  \\
	\underline{\mathbf{H}}~ \underline{\mathbf{C}}  ~ &~  
	\underline{\mathbf{H}}~ \underline{\mathbf{D}}^{(d)}  }, \\
\nonumber
\\
 \bm{\mathcal{B}}^{(s)} & = & \AR{cc}{\underline{\mathbf{B}} ~ 
	\underline{\mathbf{H}}~ \underline{\mathbf{C}} + \underline{\mathbf{A}} ~ &~  
	\underline{\mathbf{B}} ~ 
	\underline{\mathbf{H}}~ \underline{\mathbf{D}}^{(s)} + \underline{\mathbf{B}}^{(s)}  \\
	\underline{\mathbf{H}}~ \underline{\mathbf{C}}  ~ &~  
	\underline{\mathbf{H}}~ \underline{\mathbf{D}}^{(s)}  }, 
\end{eqnarray}
From the  relations in \eqref{w_explicit} one can directly infer that bi-affinity of the state-control trajectory $ \mathbf{w}$ in terms of the policy parameters 
$ \chi$ and the exogenous signals $\zeta, \bm{\epsilon}$. Fix a realization of the switching signal $ \bm{\theta}$ then the realizations of  $\mathbf{w}$
are  affine in   $\chi $ for every realization of $\zeta, \bm{\epsilon}$ and vice versa $\mathbf{w}$  is affine in  $\zeta, \bm{\epsilon}$  for fixed policy parameters $ \chi$.
This fact will enable us to specify the parameters $ \chi$ in a computationally efficient manner by solving  explicit convex optimization programs and motivates the use of POB affine laws.

In the following we will provide explicit convex programming formulations for computing POB affine control laws that
guarantee the satisfaction of the performance 
specifications and further discuss their properties. 
We start out by deriving a robust optimization result that is central to our subsequent developments. 
Consider an affine function
$$
\mathcal{M} : \mathbb{R}^{\dim(N,T)}  \rightarrow \mathbb{R}^{n_w \times n_{\zeta}+1},  
$$
a matrix convex function
$$
\mathcal{V} : \mathbb{R}^{\dim(N,T)}  \rightarrow \mathbb{R}^{n_{\zeta}+1 \times n_{\zeta}+1},  
$$
in the sense that for all $ \mu \in [0,1]$ and 
$\chi, \hat{\chi} \in \mathbb{R}^{\dim(N,T)}$,
$$
\mathcal{V}[~ \mu ~  \chi~  + ~ (1-\mu)~  \hat{\chi}]
	~ \preceq ~ \mu ~ \mathcal{V}[\chi]~  +~  (1 - \mu)~ 
\mathcal{V}[\hat{\chi}],	
$$
along with the quadratic form
\begin{equation}
\label{quadratic_inh}
\forall  \zeta \in \mathscr{D}^N,~ \gamma  ~\geq ~
\zeta_e^{\T} ~ \mathcal{V}[\chi]  ~ \zeta_e ~+~ 2 \langle \beta,  ~ 
\mathcal{M}[ \chi]~  \zeta_e \rangle, 
\end{equation}
where 
$\zeta_e^{\T} = \AR{cc}{1 ,& \zeta^{\T} }$.
Let 
$$ 
\Pi_h^\T =  \AR{cc}{ 0_{n_{\zeta} \times 1 } & I_{n_\zeta}},~ \Pi_c^\T = 
\AR{cc}{ 1   &  0_{ 1 \times n_\zeta }}.
$$ 
and write the right hand side of the quadratic inequality above in its inhomogeneous form, 
\begin{equation}
\label{quadratic_actual}
\forall   \zeta   \in  \mathscr{D}^N ,     ~ 
\gamma  ~ \geq ~  \zeta^\T~
\Delta[ \chi]  ~ 
\zeta
~ +~ 2~    \Phi[ \chi]^\T ~
\zeta  ~+~ \Psi[ \chi ],
\end{equation}
where
\begin{eqnarray*}
	\Delta[ \chi ] & = & \Pi_h^\T  ~  \mathcal{V}[\chi]  ~ \Pi_h , \\
	\Phi[ \chi]^{\T}  & = &  \Pi_c^\T ~   \mathcal{V}[\chi] ~ \Pi_h ~+ ~
	\beta^\T ~ \mathcal{M}[ \chi] ~  \Pi_h , \\
	\Psi[ \chi] & = & \Pi_c^\T ~ \mathcal{V}[\chi]  ~ \Pi_c + 2 ~  \beta^\T  ~ \mathcal{M}[ \chi] ~  \Pi_c.
\end{eqnarray*}
The following theorem involves a parameter dependent version of the $S$-lemma. 
\begin{theorem} 
	\label{Theorem 1}
    Suppose that $ ( \gamma, \chi) $ can 
	be extended by $  \lambda \in \mathbb{R}^s$ and $ X \in \mathbb{S}^{n_{\zeta}+1}$ 
	to $ \xi = (\gamma, \lambda,  X, \chi)$
	such that 	
	\begin{eqnarray}
	\label{S-lemma_all}	
	F[\xi] \succeq  0, ~ \lambda  \geq 0, ~ X  \succeq  \mathcal{V}[ \chi] 
	\end{eqnarray}
	where	
	\begin{eqnarray}
	\label{S-lemma_parameters}	
	\nonumber
	F[\xi] & = & \AR{ccc}{ S[\xi ]_{(1,1)} &  & 
		S[ \xi]_{(1,2)} \\ 
		&  &\\
		\nonumber
		S[\xi]_{(1,2)}^{\T} 	&  & S[\xi]_{(2,2)} } \\
	\nonumber
	S[\xi]_{(1,1)} & = & 
	\gamma -  \psi[X,  \chi ]  - \sum_{i=1}^{s}    [\lambda]_i \\
	\nonumber
	S[ \xi]_{(1,2)} & = & 
	-  \phi[X, \chi]^{\T} \\
	S[\xi]_{(2,2)}  & = & 
	\sum_{i=1}^{s}   [ \lambda]_i   Q_{i}  - \delta[X] 
	\end{eqnarray}	
	
	and
	\begin{eqnarray}
	\nonumber
	\delta[ X] & = & \Pi_h^\T  ~  X ~ \Pi_h ,\\
	\nonumber
	\phi[  X,\chi]^{\T}  & = & \Pi_c^\T ~   X ~ \Pi_h ~+ ~
	\beta^\T ~\mathcal{M}[ \chi] ~  \Pi_h , \\
	\psi[  X,\chi] & = & \Pi_c^\T ~  X ~ \Pi_c + 2 ~  \beta^\T  ~ \mathcal{M}[ \chi] ~  \Pi_c.
	\label{coefficients_new}
	\end{eqnarray}
	Then \eqref{quadratic_inh} is satisfied and
	the feasible set defined by 
	\eqref{S-lemma_all} is convex. 
	
	In the case where  $ ( \gamma, \chi) $ 
	can not  be extended by $  \lambda $ and $ X$ such that \eqref{S-lemma_all} is satisfied, then we know that  
	$$
	\forall    \zeta \in \mathscr{D}^N,~~~~
	\zeta_e^{\T} ~ \mathcal{V}[\chi]  ~ \zeta_e ~+~ 2 \langle \beta,  ~ 
	\mathcal{M}[ \chi]~  \zeta_e \rangle \leq~ \gamma^-
	$$
	is violated for some $  \zeta \in \mathscr{D}^N$. 
	The critical level is 
	$$
	\gamma^- = 
	\begin{cases}
	\gamma,~~~~\text{when}~ s= 1,\\
	\frac{1}{\Theta}  ~\bigg( \gamma -   \Psi[ \chi] \bigg) + 
	\Psi[ \chi] -\epsilon, ~~~\text{when}~ s > 1,
	\end{cases}
	$$
	with $ \epsilon > 0 $ arbitrary and
	$$
	\Theta = 2 ~ \ln(s+1) + 2 \sqrt{\ln(s+1)} +1.
	$$
\end{theorem}
\begin{proof}
	First we address the convexity of \eqref{S-lemma_all} in the decision variables 	
	$\xi$ .
	The constraint  $$ \lambda \geq 0 $$ is linear in 
	$  \lambda$ and therefore convex. 
	Given that $\mathcal{M}$ is affine in $ \chi$ 
	from
	\eqref{coefficients_new} one can deduce that the maps 
	$ 
	\delta[ X], \phi[ X, \chi ],
	\psi[ X,\chi]
	$
	are all affine in the their arguments and so is the map,
	$
	F[\xi].
	$
	It follows therefore that the constraint 
	$$
	F[\xi]  \succeq 0
	$$
	is convex in $\xi$. 
	As for the inequality 
	\begin{equation}
	\label{matrix_convex_ineq}
	X \succeq  \mathcal{V}[\chi],
	\end{equation}
	its convexity follows from the fact that $\mathcal{V}[\chi]$
	is matrix convex. 
	Now suppose that   $ (\gamma, \chi) $ can 
	be extended by $  \lambda $ and $ X$ such that
	\eqref{S-lemma_all} is satisfied.  From \eqref{S-lemma_all} one can deduce that 
	for any $t \in \mathbb{R}$, $ \zeta \in \mathbb{R}^{n_{\zeta}}$,
	\begin{eqnarray}
	\label{quadratic_form_lagrange_relaxed}
	\nonumber 
	\gamma t^2 & \geq &   t^2 \psi[X,  \chi ] + \sum_{i=1}^s 
	[\lambda]_i ~ ( t^2 - \zeta^\T ~   Q_{i} ~
	\zeta ) \\
	&+ & 2 ~ t ~ \phi[X,  \chi ]^{\T} ~\zeta + 
	\zeta^\T ~ \delta[X]   ~\zeta  
	\end{eqnarray}
	The above inequality implies, by taking $t=1$ and \eqref{coefficients_new} into account,  that $ 
	\forall     \zeta \in \mathscr{D}^N$,
	\begin{equation}
	\label{bounding}
	\gamma  \geq 
	\zeta_e^\T  X \zeta_e ~+~ 2 \langle \beta,  
	\mathcal{M}[ \chi]  \zeta_e \rangle, ~ \text{where}~    \zeta_e = \AR{c}{1 \\  \zeta },
	\end{equation}
	and since $ X  \succeq  \mathcal{V}[ \chi] $
	the performance specification
	\eqref{quadratic_inh} follows.
	
	We consider now the situation when the given convex problem
	\eqref{S-lemma_all} is not feasible.
	First note that given $ (\gamma, \chi )$, existence of 
	$X, \lambda$ such that \eqref{S-lemma_all}
	holds is equivalent to the statement
	\begin{equation}
	\label{S-lemma_piecemal}
\exists X \succeq  \mathcal{V}[ \chi]:	\bigg(\exists \lambda \in \mathbb{R}^{s}: ~
	F[\xi]  \succeq 0,
	\lambda \geq 0\bigg).
	\end{equation}
	If the given pair $ (\gamma, \chi) $ can not 
	be extended by $  \lambda $ and $ X$ such that
	\eqref{S-lemma_all} is satisfied then we will distinguish two cases.
	\begin{itemize}
		\item \textbf{Case 1}, $s=1$:  
		We apply the
		inhomogeneous S-lemma,   \cite{ bental6}, to the statement inside the parentheses in \eqref{S-lemma_piecemal} to write equivalently,
		$ \exists X \succeq \mathcal{V}[ \chi]:$ $\forall \zeta \in \mathscr{D}^N,$
		$$
		\label{extended_quadratic} 
		\gamma  \geq   \zeta ^\T  \delta[X]    \zeta  + 2  \phi[X_,  \chi ]^{\T} \zeta 
		+  \psi[X,  \chi ].
		$$
		By taking \eqref{coefficients_new} into account, the last statement is equivalent to 
		$\exists X \succeq \mathcal{V}[ \chi] : \forall \zeta \in \mathscr{D}^N, $ 
		$$
		\gamma  ~\geq ~
		\zeta_e^\T ~ X~ \zeta_e ~+~ 2 \langle \beta,  ~ 
		\mathcal{M}[ \chi]~  \zeta_e \rangle, 
		$$
		which in turn is equivalent  to \eqref{quadratic_inh} since 
		$\exists X \succeq \mathcal{V}[ \chi] $.  Thus 
		if a given pair $ (\gamma, \chi) $ can not 
		be extended by $  \lambda $ and $ X$ such that
		\eqref{S-lemma_all} is satisfied then \eqref{quadratic_inh} is violated for some $  \zeta \in \mathscr{D}^N$.
		
		\item  \textbf{Case 2}, $s > 1$:  Similar reasoning applies 
		to this case as well, 
		with the only difference 
		being that the  approximate inhomogeneous S-lemma, \cite{ bental6},
		is employed.  The tightness factor is adapted from \cite{jud_17}  see also Appendix A2.
		Let 
		\begin{eqnarray*}
			\SDP[\chi]  =  \min \bigg\{ \omega ~& : & ~  
			F[\omega,
			\lambda, X , \chi ] \succeq 0,  \\
			&& ~\lambda \geq 0,  
			X \succeq \mathcal{V}[ \chi]  \bigg\}.
		\end{eqnarray*}
		If $ (\gamma, \chi)$ can not be extended
		by $X, \lambda$ such that  \eqref{quadratic_inh}
		holds then we know that 
		$
		\gamma < \SDP[\chi].
		$
		Furthermore,
		\begin{eqnarray*}
			\SDP[\chi]  =  \min \bigg\{&&  \tilde{\SDP}[X, \chi]  +   \bm{\psi}_j[  X,\chi]  :   \\  	
			&& 	X \succeq V[\chi]  \bigg\},~~
		\end{eqnarray*}
		where	
		\begin{eqnarray}
		\label{derivation}
		&&	\tilde{\SDP}[X, \chi]   =  \min \bigg\{ \tau  ~:~
		\lambda \geq 0, \\
		\nonumber
		&&
		\AR{cc}{\tau - \sum_{i=1}^{s}  ~  [\lambda]_i   & -  \phi[X,  \chi ]^{\T}\\
			& \\
			-\phi[X,  \chi ]& 
			\sum_{i=1}^{s}   [ \lambda]_i   Q_{i}  -  
			\delta[X ] } \succeq 0
		\bigg \}.
		\end{eqnarray}
		Applying the  approximate inhomogeneous S-lemma to \eqref{derivation} gives
		$$
		\tilde{\SDP}[X, \chi]~ \leq~ \Theta ~\tilde{\OPT}[X, \chi]
		$$
		with
		$$
		\Theta = 2 ~ \ln(s+1) + 2 \sqrt{\ln(s+1)} +1
		$$
		and 
		$$
		\tilde{\OPT}[X, \chi] =   \max_{ \zeta \in \mathscr{D}^N }
		\bigg\{  \zeta^\T~
		X ~ 
		\zeta
		+ 2 ~ \phi[X, \chi]^\T ~
		\zeta  \bigg\}.
		$$
		Let 
		$$
		\OPT[\chi] =   \max_{ \zeta \in \mathscr{D}^N }
		\bigg\{  \zeta^\T~
		\Delta[ \chi]  ~ 
		\zeta
		+ 2 ~ \Phi[ \chi]^\T ~
		\zeta  \bigg \},
		$$
		by taking into account that  $	X \succeq V[\chi]   \succeq 0 $
		\begin{eqnarray*}
			\Theta ~ \OPT[\chi] ~+ ~ \Psi[ \chi]  =
			\min \bigg \{ && \Theta ~  \tilde{\OPT}[X, \chi]   +   \psi[  X,\chi]  ~:~    \\
			&& X \succeq  \mathcal{V}[\chi]  +  \bigg\},
		\end{eqnarray*}
		which leads to
		$$
		\gamma -\Psi[ \chi] ~<~ \Theta ~ \OPT[\chi].
		$$
		Constraint violation occurs for any $\gamma^-$ that 
		satisfies $$ \gamma^- < \OPT[\chi] + \Psi[ \chi].
		$$ 
		From the last two strict inequalities the desired result follows. 	
	\end{itemize}
\end{proof}

\noindent Theorem 1 will serve as the basis for deriving explicit semi-definite programming formulations 
for the control design problem with respect to the quadratic inequalities under consideration. 

\subsection{Averaged quadratic inequalities in  $\mathbf{w}$ }
We start out by repeating \eqref{specification_averaged_quadratic} for some fixed index $j \in \mathbb{Z}_{[1,r]}$,
\begin{equation}
\label{specification_j}
\forall   \zeta \in \mathscr{D}^N, ~
\mathbb{E}[\langle \mathcal{A}_j ~  \mathbf{w},   \mathbf{w} \rangle + 2 ~ \langle \beta_j ,  \mathbf{w} \rangle ] ~ \leq~ \gamma_j.
\end{equation}
One can write 
$$
\bm{b}~ +~
\mathbf{w}^{(d)} =  (\mathcal{F} + \mathcal{G}) ~ \zeta_e,
$$
where
$$
\mathcal{F} = 
\AR{ccccc}{ \underline{\mathbf{B}} ~ \underline{\mathbf{h}}, & &  
	\underline{\mathbf{B}}  ~ \underline{\mathbf{H}} ~ \underline{\mathbf{C}} , & &
	\underline{\mathbf{B}}  ~ \underline{\mathbf{H}} ~ \underline{\mathbf{D}}^{(d)} \\ 
	& & & &\\
	\underline{\mathbf{h}}, & & 
	\underline{\mathbf{H}} ~ \underline{\mathbf{C}} , & &
	\underline{\mathbf{H}} ~ \underline{\mathbf{D}}^{(d)} }
$$
and
$$
\mathcal{G}  =  \AR{ccccc}{ 0, & &  
	\underline{\mathbf{A}} , & &
	\underline{\mathbf{B}}^{(d)}   \\ 
	& & & &\\
	0, & & 
	0 , & &
	0 }.
$$
In view of \eqref{w_explicit} by setting 
\begin{equation}
\label{coefficient_M}
\mathcal{M} = \mathbb{E}[\mathcal{F} + \mathcal{G}], 
\end{equation}
and 
\begin{equation}
\label{coefficient_V}
\mathcal{V} =\mathbb{E}[  (\mathcal{F} + \mathcal{G})^{\T} ~  \mathcal{A}_j~ (\mathcal{F} + \mathcal{G}) ]+ \mathbb{E}[\mathcal{H}]
\end{equation}
with 
$$
\mathcal{H} = \AR{ccc}{ \trace\bigg[ \mathcal{B}^{(s)^{\T}}   \mathcal{A}_j    \mathcal{B}^{(s)} \Sigma_{\bm{\epsilon}} \bigg]   & &  0 \\ & & \\ 0 &  & 0   },
$$
\eqref{quadratic_inh} and \eqref{specification_j} are equivalent as far as their functional form is concerned. 
What remains to show is that $ \mathcal{M} $ and $\mathcal{V}$ defined in \eqref{coefficient_M}  and \eqref{coefficient_V} are affine and matrix convex respectively.

\begin{proposition}
$ \mathcal{M} $ and $\mathcal{V}$ defined in \eqref{coefficient_M}  and \eqref{coefficient_V} are affine and matrix convex in $\chi$ respectively. 
\end{proposition}

\begin{proof}
Given the statistical description of the switching process via $ \pi, P  $ one has 
for any $ 0 \leq \tau \leq N-1$,
\begin{equation}
\label{probability_path}
\Pr[ (\theta_0 , \theta_1, \hdots , \theta_\tau)] = 
\bigg(\prod_{t=0}^{\tau-1} 
[P]_{\theta_{t+1}, \theta_{t}}\bigg) ~ \pi_{\theta_0} .
\end{equation}
Furthermore, from \eqref{coefficient_M} one has
$$
\mathcal{M} = \sum_{\theta
	\in \mathbb{O}^N} ~ \Pr[\theta ] ~ (\mathcal{F} + \mathcal{G}).
$$
Given that $ \mathcal{F}$ is linear in $\chi$ and $\mathcal{G}$ is independent of $\chi$ it follows that
$ \mathcal{M}$ is affine in $\chi$.
As for $\mathcal{V}$ first note that for every realization of $\theta \in \mathbb{O}^N$ $ \mathcal{F} + \mathcal{G}$ 
as well as $ \mathcal{H}$ are matrix convex in $ \chi$. We demonstrate the calculation for $ \mathcal{F} + \mathcal{G}$ 
similar reasoning applies to $\mathcal{H}$.
Let $ \chi , \hat{\chi} \in \mathcal{X}$ and $ \mu \in [0,1]$ and set $ \mathcal{J} = (\mathcal{F} + \mathcal{G})^{\T}  \mathcal{A}_j (\mathcal{F} + \mathcal{G})$, then
\begin{eqnarray*}
	\mathcal{R}& = & 
	\mu ~ \mathcal{J}[\chi] ~+~ (1-\mu)~ \mathcal{J}[\hat{\chi}] - 
	\mathcal{J}[\mu ~ \chi ~+~ (1-\mu)~ \hat{\chi}] 	\\
	& = & 
	\mu ~ (\mathcal{F}[\chi] + \mathcal{G})^\T ~ \mathcal{A}_j ~ 
	(\mathcal{F}[\chi] + \mathcal{G}) + (1-\mu) ~ 
	(\mathcal{F}[\hat{\chi}] + \mathcal{G})^\T ~ \mathcal{A}_j ~ 
	(\mathcal{F}[\hat{\chi}] + \mathcal{G}) \\
	&& - (\mathcal{F}[\mu  \chi + (1- \mu) \hat{\chi}  ] + \mathcal{G})^\T ~ 
	\mathcal{A}_j ~ 
	(\mathcal{F}[\mu  \chi + (1- \mu) \hat{\chi}  ] + \mathcal{G}) \\
	& = & \mu ~ (1 - \mu) ~ \bigg( \mathcal{F}[\chi]^\T ~  \mathcal{A}_j ~ \mathcal{F}[\chi] - \mathcal{F}[\chi]^\T ~  \mathcal{A}_j ~ \mathcal{F}[\hat{\chi}]  -  \mathcal{F}[\hat{\chi}]^\T ~  \mathcal{A}_j ~ \mathcal{F}[\chi]  +   \mathcal{F}[\hat{\chi}]^\T ~  \mathcal{A}_j ~ \mathcal{F}[\hat{\chi}]\bigg) \\
	&=& \mu ~ (1 - \mu) ~ \mathcal{F}[\chi - \hat{\chi}]^\T ~ \mathcal{A}_j ~\mathcal{F}[\chi - \hat{\chi}] \\
	&\succeq& 0
\end{eqnarray*}
From \eqref{coefficient_V} one has 
$
\mathcal{V}  = \sum_{\theta 
	\in \mathbb{O}^N} ~ \Pr[\theta ] ~ \mathcal{J}
$
and therefore $ \mathcal{V} $ is matrix convex in $ \chi$. 
\end{proof}
Applicability of theorem 1 is established and as such we have arrived to an explicit semidefinite programing formulation for designing the control parameters in regards to the specification
\eqref{specification_j}.

\subsection{Quadratic inequalities in  $\mu_{\mathbf{w}}$ }

We start out by repeating \eqref{specification_mean_quadratic}  for some fixed index $j \in \mathbb{Z}_{[1,\hat{r}]}$,
\begin{equation}
\label{specification_j_mean}
\forall   \zeta \in \mathscr{D}^N, ~ \langle \hat{\mathcal{A}}_j ~  \mu_\mathbf{w},   \mu_\mathbf{w} \rangle + 2 ~ \langle \hat{\beta}_j ,  \mu_\mathbf{w} \rangle  ~ \leq~ \hat{\gamma}_j, 
\end{equation}
In view of \eqref{w_explicit} and \eqref{coefficient_M} one has  
$$
\mu_\mathbf{w} = \mathcal{M} ~ \zeta_e,
$$
and by setting 
$$
\mathcal{V} = \mathcal{M}^{\T} ~ \hat{\mathcal{A}}_j ~ \mathcal{M},
$$
\eqref{quadratic_inh} and \eqref{specification_j_mean} are equivalent as far as their functional form is concerned. 
Matrix convexity of $\mathcal{V}$ follows by similar arguments as in proposition 2. As such applicability of theorem 1 is valid in this case as well  
and therefore we have arrived to an explicit semidefinite programing formulation for designing the control parameters in regards to the specification
\eqref{specification_j_mean}.

\subsection{Robust steering of densities }

For this problem we assume that $m=1$ thus stochasticity is induced by the external noise in the absence of Markovian-switching.
The quadratic inequality on $\mu_w$ has already been discussed we will focus on the semidefinite constraint on the covariance matrix. 
Given that $ \mathbb{E}[\bm{\epsilon}] = 0$, one has
\begin{equation}
\label{mean_formula}
\mu_{\mathbf{w}} = \bm{b}
~ + ~ \mathbf{w}^{(d)},~~ \mathbf{w} - \mu_{\mathbf{w}} = 
\mathbf{w}^{(s)},
\end{equation}
and given that the basic random variables $ \mathbf{s}_0, 
\{ \mathbf{e} \}_0^{N-1}
$
are jointly independent it follows that  
$$
\Sigma_{\bm{\epsilon}} = \mathbb{E}[ \bm{\epsilon} \bm{\epsilon}^{\T}] = \AR{cc}{\Sigma_0 & 0 \\  0 & I}
$$
and therefore 
\begin{equation}
\label{covariance_formula}
\Sigma_{\mathbf{w}} =  \bm{\mathcal{B}}^{(s)} ~ 
\Sigma_{\bm{\epsilon}}~ \bm{\mathcal{B}}^{{(s)}^{\T}}.
\end{equation}

\begin{proposition}
	Fix $j \in \mathbb{Z}_{[1,\tilde{r}]}$.
	The policy associated to the control parameters 
	$ \chi$  
	leads to a closed loop system where the specification under consideration is satisfied robustly, i.e.
	\begin{equation}
	\label{specification_j_covariance}
	\forall   \zeta \in \mathscr{D}^N, ~  \mathcal{Q}_j  ~ \Sigma_\mathbf{w}  ~  \mathcal{Q}_j^\T \preceq \tilde{\Sigma}_j
	\end{equation}
if and only if 
\begin{equation}
\label{steering_covariance}
 \mathcal{Q}_j  ~ \bm{\mathcal{B}}^{(s)} ~ 
 \Sigma_{\bm{\epsilon}}~ \bm{\mathcal{B}}^{{(s)}^{\T}}  ~  \mathcal{Q}_j^\T \preceq \tilde{\Sigma}_j
\end{equation}
	The feasible set defined by 
	\eqref{steering_covariance} is convex. 

\end{proposition}
\begin{proof}
	The proof is immediate, using \eqref{covariance_formula}
we establish the equivalence of \eqref{specification_j_covariance} and \eqref{steering_covariance}. The left hand side of
\eqref{steering_covariance} 
$$
\mathcal{Q}_j  ~ \bm{\mathcal{B}}^{(s)} ~ 
\Sigma_{\bm{\epsilon}}~ \bm{\mathcal{B}}^{{(s)}^{\T}}  ~  \mathcal{Q}_j^\T
$$	
is matrix convex in $\chi$ by similar arguments as in the proof of proposition 2 and as such the 
feasible set of  \eqref{steering_covariance} is convex. 	
\end{proof}
In retrospect we recognize that 
the choice of the uncertainty set for the initial state and set bounded disturbances is dictated by the computational machinery we employ in order to 
arrive at a tractable design procedure. 
It was this particular choice of $ \mathscr{D}^N$ that enabled the exploitation of  
the bi-affinity of the state-control trajectory induced by the POB control laws during the subsequent use of the approximate inhomogeneous $S$-lemma.

\subsection{Polynomial-time Computation for Linear and Quadratic Forms }
\label{sec:52}
In this subsection we  provide recursive computation procedures  for the terms appearing in the linear
form $ \mathcal{M}$ and the quadratic form 
$  \mathcal{V}$,  that and enter in \eqref{S-lemma_all}
as part of  our convex programming formulation.
We show that for a fixed $T$
the coefficients of the functions $\mathcal{M}$  and $\mathcal{V}$ 
can be computed in polynomial in  $N, m$  number  of arithmetic operations.
To avoid overly long expressions we  introduce some notation.
Given $ \tau, T \in \mathbb{N}$ and two sequences of positive integers, $ n_0, n_1, \hdots, n_N \in \mathbb{Z}_+$,
$ m_0, m_1, \hdots, m_N \in \mathbb{Z}_+$ define the matrix valued functions
\begin{eqnarray*}
	\mathbf{f}_i  :  \mathbb{O} \rightarrow \mathbb{R}^{n_{i+1} \times n_i}, ~~
	\mathbf{q}_i  :  \mathbb{O} \rightarrow \mathbb{R}^{n_{i+1} \times n_i}, 
\end{eqnarray*}
where $i = 0, 1, \hdots, \tau-1, \tau+1, \tau+2, 
\hdots N-1$. For $ t \geq s > \tau$ and $ \tau > t \geq s$ let 
\begin{eqnarray*}
	\mathbf{F}_{t,s}  :  \mathbb{O}^{t-s+1} \rightarrow \mathbb{R}^{n_{t+1} \times n_s}, ~~
	\mathbf{Q}_{t,s}  :  \mathbb{O}^{t-s+1} \rightarrow \mathbb{R}^{n_{t+1} \times n_s}, 
\end{eqnarray*}
with 
\begin{eqnarray*}
	\mathbf{F}_{t,s}[\theta_{t}, \hdots, \theta_{s}] & = & \mathbf{F}_{t, \tau}[\bm{\theta}_{[t,t-s]}]
	= \mathbf{f}_{t}[\theta_{t}] ~\mathbf{f}_{t-1}[\theta_{t-1}] ~ \hdots ~ \mathbf{f}_{\tau}[\theta_{s}], \\
	\mathbf{Q}_{t,s}[\theta_{t}, \hdots, \theta_{s}] & = & \mathbf{Q}_{t, \tau}[\bm{\theta}_{[t,t-s]}]
	= \mathbf{q}_{t}[\theta_{t}] ~\mathbf{q}_{t-1}[\theta_{t-1}] ~ \hdots ~ \mathbf{q}_{\tau}[\theta_{s}].
\end{eqnarray*}
We first discuss the computation of $\mathcal{M} $.
The entries of $\mathcal{M} $ are sums of polynomially many terms of the form 
\begin{equation}
\label{linear_term}
\mathbb{E}\bigg[\mathbf{f}_{N-1}[\theta_{N-1}] ~\mathbf{f}_{N-2}[\theta_{N-2}] ~ \hdots ~ \mathbf{f}_{\tau+1}[\theta_{\tau+1}] ~ \mathbf{f}_{\tau}[\bm{\theta}_{[\tau,T]}] ~ 
\mathbf{f}_{\tau-1}[\theta_{\tau-1}]  ~ \hdots ~ \mathbf{f}_{\tau-T}[\theta_{\tau-T}]   ~ \hdots ~
\mathbf{f}_0[\theta_0]\bigg],
\end{equation}
where $\mathbf{f}_0, \hdots, \mathbf{f}_{N-1}$ are matrix valued functions of conformable dimensions 
so that the product makes sense. The actual entries are given explicitely in the appendix. 
While the number of elements in the set $ \mathbb{O}^N$ grows exponentially in $N$ 
we will show in the following proposition that the complexity of computing each term of the form \eqref{linear_term}
grows polynomially in $N$ for a fixed $T$.  The proof of this result is given in Appendix C.

\begin{proposition} \label{lemma1}
	Consider a Markov chain $ \bm{\theta} \in \mathbb{O}^N$ with statistical description given via
	$ \pi, P$ 
	and the matrix valued expression 
	\begin{equation}
	\label{typical_linear_term}
	\mathbf{T}_l  = \mathbb{E}\bigg[\mathbf{F}_{N-1, \tau+1}[\bm{\theta}_{[N-1,N-\tau-2]}] ~ \mathbf{f}_{\tau}[\bm{\theta}_{[\tau,T]}] ~ 
	\mathbf{F}_{\tau-1, 0}[\bm{\theta}_{[\tau-1,\tau-1]}]\bigg].
	\end{equation}
	Define the matrix valued functions 
	\begin{eqnarray}
	\label{recursive_linear_1}
	\mathbf{g}_{N-2}[\theta_{N-2}] & = & \sum_{\theta_{N-1} \in \mathbb{O}} 
	[P]_{\theta_{N-1}, \theta_{N-2}} \mathbf{f}_{N-1}[\theta_{N-1}],  \\
	\label{recursive_linear_2}
	\mathbf{g}_{s}[\theta_{s}] & = & \sum_{\theta_{s+1} \in \mathbb{O}} 
	[P]_{\theta_{s+1}, \theta_{s}} ~  \mathbf{g}_{s+1}[\theta_{s+1}]~ \mathbf{f}_{s+1}[\theta_{s+1}], ~~~
	s = N-3, N-4, \hdots, \tau+1, \tau,  \\
	\label{recursive_linear_3}
	\mathbf{g}_{\tau-T-1}[\theta_{\tau-T-1}] & = &
	\sum_{(\theta_{\tau-T},\theta_{\tau-T+1}, \hdots, \theta_{\tau} )} 
	[P]_{\theta_{\tau}, \theta_{\tau-1}} \hdots [P]_{\theta_{\tau-T}, \theta_{\tau-T-1}}
	~\mathbf{g}_{\tau}[\theta_{\tau}] ~ 
	\mathbf{F}_{\tau-1, \tau-T}[\bm{\theta}_{[\tau-1,T-1]}], \\
	\label{recursive_linear_4}
	\mathbf{g}_{s}[\theta_{s}] & = & \sum_{\theta_{s+1} \in \mathbb{O}} 
	[P]_{\theta_{s+1}, \theta_{s}} ~  \mathbf{g}_{s+1}[\theta_{s+1}]~ \mathbf{f}_{s+1}[\theta_{s+1}], ~~~
	s = \tau - T -2 ,  \tau - T - 3, \hdots, 0.  
	\end{eqnarray}
	It follows that 
	$$
	\mathbf{T}_l = \sum_{\theta_0 \in \mathbb{O}}	\Pr[ \theta_0 ] ~ \mathbf{g}_0[\theta_0] ~\mathbf{f}_0[\theta_0].
	$$
\end{proposition}

\vspace{0.1in}

Without utilizing Proposition~\ref{lemma1} computing $\mathbf{T}_l$ via
$$
\mathbf{T}_l = \sum_{\bm{\theta} \in \mathbb{O}^N}~ 	\Pr[\bm{\theta} ]~ 
\mathbf{F}_{N-1, \tau+1}[\bm{\theta}_{[N-1,N-\tau-2]}] ~ \mathbf{f}_{\tau}[\bm{\theta}_{[\tau,T]}] ~ 
\mathbf{F}_{\tau-1, 0}[\bm{\theta}_{[\tau-1,\tau-1]}]
$$
requires building $m^N$ matrix products of the form $\mathbf{F}_{N-1, \tau+1}[\bm{\theta}_{[N-1,N-\tau-2]}] ~ \mathbf{f}_{\tau}[\bm{\theta}_{[\tau,T]}] ~ 
\mathbf{F}_{\tau-1, 0}[\bm{\theta}_{[\tau-1,\tau-1]}]$. 
Each long product requires the multiplication of $N$ matrices. On the other hand utilization of Proposition~\ref{lemma1} requires 
multiplication of 2 matrices $m  (N-T)$ times as well as multiplication of $T+1$ matrices $m^{(T+1)}$ times, making the computation of  the coefficients of the affine function 
$ \mathcal{M}$  polynomial in $N, m$.

Now let us consider the computation of 
$
\mathcal{V},
$ 
which follows a similar route to that of $\mathcal{M} $. The only difference is that 
one needs to consider quadratic expressions. 
Partition for the fixed $j \in \mathbb{Z}_{[1,r]}$, $\mathcal{A}_j$ as 
$$
\mathcal{A}_j = \AR{cc}{\mathcal{A}_{xx,j} &  \mathcal{A}_{xu,j} \\ 
	\mathcal{A}_{ux,j} & \mathcal{A}_{uu,j}},~~~\text{where}~~~ 
\mathcal{A}_{ux,j}^\T = \mathcal{A}_{xu,j},
$$
from \eqref{w_explicit} and \eqref{coefficient_V} it follows that
$$
\mathcal{V} = \AR{lll}{\mathcal{V}_{(I,I)} & \mathcal{V}_{(I,II)} & 
	\mathcal{V}_{(I,III)} \\ 
	\mathcal{V}_{(II,I)} &\mathcal{V}_{(II,II)} & \mathcal{V}_{(II,III)} \\
	\mathcal{V}_{(III,I)} &\mathcal{V}_{(III,II)} & \mathcal{V}_{(III,III)}}
$$

The entries of $\mathcal{V}  $ are sums of polynomially many terms of the form 
\begin{eqnarray}
\label{quadratic_term}
\nonumber
&& \mathbb{E}\bigg[ 
\bigg( \mathbf{Q}_{N-1, \tau+1}[\bm{\theta}_{[N-1,N-\tau-2]}] ~ 
\mathbf{q}_{\tau}[\bm{\theta}_{[\tau,T]}] ~ 
\mathbf{Q}_{\tau-1, 0}[\bm{\theta}_{[\tau-1,\tau-1]}]
\bigg)^\T ~ \mathbf{S} ~\\
&& \bigg( \mathbf{F}_{N-1, \tau+1}[\bm{\theta}_{[N-1,N-\tau-2]}] ~ 
\mathbf{f}_{\tau}[\bm{\theta}_{[\tau,T]}] ~ 
\mathbf{F}_{\tau-1, 0}[\bm{\theta}_{[\tau-1,\tau-1]}]
\bigg) \bigg]
\end{eqnarray}
where $\mathbf{f}_0, \hdots, \mathbf{f}_{N-1}$ and $\mathbf{q}_0, \hdots, \mathbf{q}_{N-1} $ are matrix valued functions of conformable dimensions 
so that the product makes sense, while $\mathbf{S}$ is a constant matrix. 
While the number of elements in the set $ \mathbb{O}^N$ grows exponentially in $N$ 
we will show in the following  proposition that the complexity of computing each term of the form \eqref{quadratic_term}
terms grows polynomially in $N$ for a fixed $T$. Again the proof of this result is relegated to Appendix D.

\begin{proposition} \label{lemma2}
	Consider a Markov chain $ \bm{\theta} \in \mathbb{O}^N$ with statistical description given via $ \pi$, $P$ 
	and the matrix valued expression 
	\begin{eqnarray}
	\label{typical_quadratic_term}
	\nonumber
	\mathbf{T}_q  =  \mathbb{E}\bigg[&&
	\bigg( \mathbf{Q}_{N-1, \tau+1}[\bm{\theta}_{[N-1,N-\tau-2]}] ~ 
	\mathbf{q}_{\tau}[\bm{\theta}_{[\tau,T]}] ~ 
	\mathbf{Q}_{\tau-1, 0}[\bm{\theta}_{[\tau-1,\tau-1]}]
	\bigg)^\T ~ \mathbf{S} ~ \\
	&&  \bigg( \mathbf{F}_{N-1, \tau+1}[\bm{\theta}_{[N-1,N-\tau-2]}] ~ 
	\mathbf{f}_{\tau}[\bm{\theta}_{[\tau,T]}] ~ 
	\mathbf{F}_{\tau-1, 0}[\bm{\theta}_{[\tau-1,\tau-1]}]
	\bigg) \bigg] 
	\end{eqnarray}
	where  $\mathbf{f}_0, \hdots, \mathbf{f}_{N-1}, \mathbf{q}_0, \hdots, \mathbf{q}_{N-1}$ 
	are matrix valued functions of conformable dimensions so that the product makes sense, while $ \mathbf{S}$ is a constant matrix.  
	Define the matrix valued functions 
	\begin{eqnarray}
	\label{recursive_quadratic}
	\mathbf{g}_{N-2}[\theta_{N-2}] & = & \sum_{\theta_{N-1} \in \mathbb{O}} 
	[P]_{\theta_{N-1}, \theta_{N-2}} ~ \mathbf{q}_{N-1}[\theta_{N-1}]^\T~
	\mathbf{S} ~  \mathbf{f}_{N-1}[\theta_{N-1}],  \\
	\nonumber
	\mathbf{g}_{s}[\theta_{s}] & = & \sum_{\theta_{s+1} \in \mathbb{O}} 
	[P]_{\theta_{s+1}, \theta_{s}} ~  
	~ \mathbf{q}_{s+1}[\theta_{s+1}]^\T~
	\mathbf{g}_{s+1}[\theta_{s+1}]~ \mathbf{f}_{s+1}[\theta_{s+1}], \\ ~
	&& s = N-3, N-4, \hdots, \tau+1, \tau,  \\
	\nonumber
	\mathbf{g}_{\tau-T-1}[\theta_{\tau-T-1}] & = &
	\sum_{(\theta_{\tau-T},\theta_{\tau-T+1}, \hdots, \theta_{\tau} )} 
	[P]_{\theta_{\tau}, \theta_{\tau-1}} \hdots [P]_{\theta_{\tau-T}, \theta_{\tau-T-1}} ~ \bigg( \mathbf{q}_{\tau}[\bm{\theta}_{[\tau,T]}] ~ \mathbf{Q}_{\tau-1, \tau-T}[ \bm{\theta}_{[\tau-1,T-1]}] \bigg)^\T,\\
	&& 
	~ \mathbf{g}_{\tau}[\theta_{\tau}] ~
	\bigg( \mathbf{f}_{\tau}[\bm{\theta}_{[\tau,T]}] ~ 
	\mathbf{F}_{\tau-1, \tau-T}[ \bm{\theta}_{[\tau-1,T-1]}] \bigg), \\
	\mathbf{g}_{s}[\theta_{s}] & = & \sum_{\theta_{s+1} \in \mathbb{O}} 
	[P]_{\theta_{s+1}, \theta_{s}} ~  
	~ \mathbf{q}_{s+1}[\theta_{s+1}]^\T~
	\mathbf{g}_{s+1}[\theta_{s+1}]~ \mathbf{f}_{s+1}[\theta_{s+1}], ~~~
	s = \tau - T -2 ,   \hdots, 0.  
	\end{eqnarray}
	It follows that 
	$$
	\mathbf{T}_q = \sum_{\theta_0 \in \mathbb{O}}	\Pr[ \theta_0 ] ~ \mathbf{q}_0[\theta_0]^\T ~ \mathbf{g}_0[\theta_0] ~\mathbf{f}_0[\theta_0].
	$$
\end{proposition}

\vspace{0.1in}

Without utilizing Proposition~\ref{lemma2} computing $\mathbf{T}_q$ via
\begin{eqnarray*}
	\mathbf{T}_q  =  \sum_{\bm{\theta} \in \mathbb{O}^N}~ 	\Pr[\bm{\theta} ]~
	&&\bigg( \mathbf{Q}_{N-1, \tau+1}[\bm{\theta}_{[N-1,N-\tau-2]}] ~ 
	\mathbf{q}_{\tau}[\bm{\theta}_{[\tau,T]}] ~ 
	\mathbf{Q}_{\tau-1, 0}[\bm{\theta}_{[\tau-1,\tau-1]}]
	\bigg)^\T ~ \mathbf{S} ~ \\
	&&  \bigg( \mathbf{F}_{N-1, \tau+1}[\bm{\theta}_{[N-1,N-\tau-2]}] ~ 
	\mathbf{f}_{\tau}[\bm{\theta}_{[\tau,T]}] ~ 
	\mathbf{F}_{\tau-1, 0}[\bm{\theta}_{[\tau-1,\tau-1]}]
	\bigg) 
\end{eqnarray*}
requires building $m^N$ matrix products of the form 
$$ \bigg( \mathbf{Q}_{N-1, \tau+1}[\bm{\theta}_{[N-1,N-\tau-2]}] ~ 
\mathbf{q}_{\tau}[\bm{\theta}_{[\tau,T]}] ~ 
\mathbf{Q}_{\tau-1, 0}[\bm{\theta}_{[\tau-1,\tau-1]}]
\bigg)^\T ~ \mathbf{S} ~ \\
\bigg( \mathbf{F}_{N-1, \tau+1}[\bm{\theta}_{[N-1,N-\tau-2]}] ~ 
\mathbf{f}_{\tau}[\bm{\theta}_{[\tau,T]}] ~ 
\mathbf{F}_{\tau-1, 0}[\bm{\theta}_{[\tau-1,\tau-1]}]
\bigg). $$
Each long product requires the multiplication of $2~N~+1~$ matrices. 
On the other hand utilization of Proposition~\ref{lemma2} requires 
multiplication of 3 matrices $m (N-T)$ times as well as multiplication of $2  ~T +  3$ matrices $m^{(T+1)}$ times, making
the computation of the coefficients of $ \mathcal{V}  $ polynomial in $N,m$.

\subsection{Relation between POB and OB affine control laws}

Our motivation for using POB affine control laws is computational tractability. 
We will discuss the relation to the traditional affine policy based on this actual outputs of the system. 
To this end consider the classical output-affine policy 
\begin{equation}
\label{control_out}
\mathbf{u}_t = g_{t}[\bm{\theta}_{[t,T]}] + \sum_{i=0}^t
G_i^t[\bm{\theta}_{[t,T]}]~ \mathbf{y}_i,~~
~t \in  \mathbb{Z}_{[0,N-1]},
\end{equation}
where
\begin{equation}
\label{parameters_out_control}
g_{t} :  \mathbb{O}^{\min{\{T+1,t+1\}}}  \rightarrow \mathbb{R}^{n_u},~ G_j^t :    \mathbb{O}^{\min{\{T+1,t+1\}}} \rightarrow \mathbb{R}^{n_u \times n_x},
\end{equation}
$0 \leq j \leq t, ~~ t \in \mathbb{Z}_{[0,N-1]}$.

When employing the classical output-based (OB) affine policy  $\{ g_{t}, G_j^t \}_{j \leq t, t \leq N-1} $ 
the state-control trajectory is affine 
in the initial state and disturbances while being highly nonlinear in the parameters of the policy \cite{bental2}.
The very first question 
is which of them is more flexible, as far as achievable behavior of the controlled system is concerned. 
The equivalence between POB and OB affine policies for deterministic linear time varying systems was established in 
\cite{bental2} in terms of the respective families of achievable closed loop state-control trajectories. 
A similar result can be derived in the case of Markovian switching with one important caveat. 
For Markov jump linear systems
the equivalence between the affine OB and affine POB policies assumes that both policies can depend, 
in a non-anticipative manner, on the entire sequence of the switching sequence. In order to avoid exponential exploding in the number of parameters specifying the policies as N grows, 
we, normally  restrict the ``switching memory'' $T$ of the policies, by restricting their parameters to depend on (at most) a fixed moderate number of the latest values of the switching signal. 
With this restriction, the two classes of affine policies in question are, in general, not equivalent  and the respective families of achievable state-control trajectories are in a ``general position'' to each other. The formal statement follows.
\begin{theorem} 
	\label{Theorem 1}
	Consider a  realization of the initial state and disturbances 
	$  \zeta \in\mathscr{D}^N$ and let $T = N-1$.
	\begin{itemize}	
		\item [(i)] For every affine OB policy  there exists an POB affine policy  
		 such that the resulting respective state-control trajectories 
		of the associated closed-loop systems are the same. 
		\item [(ii)] Vice versa, for every POB affine policy there exists an affine  OB policy such that the resulting respective state-control trajectories 
		of the associated closed-loop systems are the same. 
	\end{itemize}
\end{theorem}

\textbf{Proof:}
The purified outputs are the outputs of a dynamical model that is driven solely by the initial state and disturbances, has state vector
$ \bm{\delta}_t = \mathbf{x}_t - \hat{\mathbf{x}}_t $, $ t =0,1, \hdots, N-1,$ and  state space realization
\begin{eqnarray*}
\nonumber
\bm{\delta}_0 & = & \mathbf{z},\\
\bm{\delta}_{t+1} & = & \mathbf{A}_t[\theta_t] ~ \bm{\delta}_t + \mathbf{R}_t[\theta_t] ~ \mathbf{d}_t,   \\
\nonumber
\mathbf{v}_t  & = & \mathbf{C}_t[\theta_t] ~ \bm{\delta}_t +  \mathbf{D}_t[\theta_t] ~ \mathbf{d}_t,~~~  
t =0,1, \hdots, N-1.
\end{eqnarray*}

Let 
$$\mathbf{y}^\T = \AR{cccc}{\mathbf{y}_0^\T, \mathbf{y}_1^\T, \hdots, \mathbf{y}_{N-1}^\T}^\T \in \mathbb{R}^{n_y  N},$$
denote the trajectory 
the actual outputs.
Depending on the particular choice of  policy one can write
\begin{equation}
\label{ vector_u_pure_out}
\mathbf{u} =\underline{\mathbf{h}} ~+~ \underline{\mathbf{H}} ~ \mathbf{v}
\end{equation}
for an affine policy in the purified outputs and 
\begin{equation}
\label{ vector_u_out}
\mathbf{u} =\underline{\mathbf{g}} ~+~ \underline{\mathbf{G}} ~ \mathbf{y}
\end{equation}
for an affine policy in the actual outputs of the system, 
where 
$$
\underline{\mathbf{g}} = \AR{c}{g_0[\bm{\theta}_{[0,T]}] \\  g_1[\bm{\theta}_{[1,T]}] \\ \vdots \\ g_{N-1}[\bm{\theta}_{[N-1,T]}]}
$$
and
$\underline{\mathbf{G}} \in \mathbb{R}^{n_u  N \times n_y  N} $ consist of $N^2$ matrices of size $ n_u \times n_y$.
The $(k,j)$-th block is
$$
[\underline{\mathbf{G}}]_{k,j} = 	\begin{cases}
G_j^k[\bm{\theta}_{[k,T]}], ~& k \geq j, \\
~~~~~0_{n_u \times n_y}, ~& k < j,~~ k,j \in \mathbb{Z}_{[0,N-1]}.
\end{cases}
$$	
(i) Fix a particular affine policy as in \eqref{control_out}. It suffices to show that for
properly chosen maps 
\begin{equation}
\mathbf{q}_{t} :  \mathbb{O}^{\min{\{T+1,t+1\}}} \rightarrow \mathbb{R}^{n_y},~~~~ 
\mathbf{Q}_j^t :  \mathbb{O}^{\min{\{T+1,t+1\}}} \rightarrow \mathbb{R}^{n_y \times n_y}, ~~~~j \leq t, ~
t =0,1, \hdots, N-1,
\end{equation}
one has 
\begin{equation}
\label{pure2out}
\mathbf{y}_t = \mathbf{q}_{t}[\bm{\theta}_{[t,T]}] + \mathbf{Q}_0^t[\bm{\theta}_{[t,T]}] ~ \mathbf{v}_0 + 
\mathbf{Q}_1^t[\bm{\theta}_{[t,T]}] ~ \mathbf{v}_1 + \hdots + \mathbf{Q}_t^t[\bm{\theta}_{[t,T]}] ~ \mathbf{v}_t, ~~ t =0,1, \hdots N-1 .
\end{equation}
The above relations together with the fixed affine laws would imply that
\begin{equation}
\label{equivalence_laws}
\mathbf{u} \equiv \underline{\mathbf{g}} ~+~ \underline{\mathbf{G}} ~ \mathbf{y} \equiv 
\underline{\mathbf{h}} ~+~ \underline{\mathbf{H}} ~ \mathbf{v}
\end{equation}
for appropriately chosen parameters $\{ h_{t}, H_j^t \}_{j \leq t,0 \leq t \leq N-1}$ 
so that the affine policy under consideration can be indeed 
represented as an affine policy via purified outputs. 

The proof of \eqref{pure2out} proceeds by induction. 
Given that $ \bm{\delta}_0  =  \mathbf{z} = \mathbf{x}_0$ it follows that 
$ \mathbf{v}_0 = \mathbf{y}_0$ proving the base case. 
Now let $ s \geq 1$ and assume that \eqref{pure2out} is valid $ 0 \leq t < s$. 
From \eqref{equivalence_laws} and \eqref{model} it follows that  $\hat{\mathbf{x}}_s$ and as such  $\hat{\mathbf{y}}_s$
are affine functions of $ \mathbf{v}_0, \mathbf{v}_1, \hdots, \mathbf{v}_{s-1},$ in the sense that for some 
$$
\mathbf{l}_{t} :  \mathbb{O}^{\min{\{T+1,t+1\}}} \rightarrow \mathbb{R}^{n_y},~~~~ 
\mathbf{L}_j^t :  \mathbb{O}^{\min{\{T+1,t+1\}}} \rightarrow \mathbb{R}^{n_y \times n_y},
~~~~j \leq t, ~ t =0,1, \hdots, N-1,
$$
one has 
$$
\hat{\mathbf{y}_s} = \mathbf{l}_{s} + 
\mathbf{L}_0^{s-1}[\bm{\theta}_{[s-1,T]}] ~ \mathbf{v}_0 + 
\mathbf{L}_1^{s-1}[\bm{\theta}_{[s-1,T]}] ~ \mathbf{v}_1 + \hdots + \mathbf{L}_{s-1}^{s-1}[\bm{\theta}_{[s-1,T]}] ~ \mathbf{v}_{s-1}.
$$
Given that $\mathbf{y}_t = \hat{\mathbf{y}_t} + \mathbf{v}_t$ \eqref{pure2out}
follows and therefore the induction is completed proving (i). \\

The proof of (ii) is similar.   Fix a particular affine policy as in \eqref{control_pure_out}. It suffices to show that for
properly chosen maps 
\begin{equation}
\mathbf{q}_{t} :  \mathbb{O}^{\min{\{T+1,t+1\}}} \rightarrow \mathbb{R}^{n_y},~~~~ 
\mathbf{Q}_j^t :  \mathbb{O}^{\min{\{T+1,t+1\}}} \rightarrow \mathbb{R}^{n_y \times n_y}
~~~~j \leq t, ~ t =0,1, \hdots, N-1,
\end{equation}
one has 
\begin{equation}
\label{out2pure}
\mathbf{v}_t = \mathbf{q}_{t}[\bm{\theta}_{[t,T]}] + \mathbf{Q}_0^t[\bm{\theta}_{[t,T]}] ~ \mathbf{y}_0 + 
\mathbf{Q}_1^t[\bm{\theta}_{[t,T]}] ~ \mathbf{y}_1 + \hdots + \mathbf{Q}_t^t[\bm{\theta}_{[t,T]}] ~ \mathbf{y}_t, ~~ t =0,1, \hdots N-1 .
\end{equation}

The above relations together with the fixed affine law as in \eqref{control_pure_out} would imply that
for appropriately chosen parameters $ \{ g_{t}, G_j^t \}_{j \leq t,0 \leq t \leq N-1}$ 
the same state-control trajectory can be produced by
an affine policy that is a function of the outputs since 
\eqref{equivalence_laws} will hold. 

Again the proof of \eqref{out2pure} proceeds by induction. 
Given that $ \bm{\delta}_0  =  \mathbf{z} = \mathbf{x}_0$ it follows that
$ \mathbf{v}_0 = \mathbf{y}_0$ proving the base case. Now let $ s \geq 1$ and assume that \eqref{out2pure} is valid $ 0 \leq t < s$. 
From \eqref{equivalence_laws} and \eqref{model} it follows that  $\hat{\mathbf{x}}_s$ and as such  $\hat{\mathbf{y}}_s$
are affine functions of $ \mathbf{y}_0, \mathbf{y}_1, \hdots, \mathbf{y}_{s-1},$ in the sense that for some 
$$
\mathbf{l}_{t} :  \mathbb{O}^{\min{\{T+1,t+1\}}} \rightarrow \mathbb{R}^{n_y},~~~~ 
\mathbf{L}_j^t :  \mathbb{O}^{\min{\{T+1,t+1\}}} \rightarrow \mathbb{R}^{n_y \times n_y},
~~~~j \leq t, ~ t =0,1, \hdots, N-1,
$$
one has 
$$
\hat{\mathbf{y}_s} = \mathbf{l}_{s} + 
\mathbf{L}_0^{s-1}[\bm{\theta}_{[s-1,T]}] ~ \mathbf{y}_0 + 
\mathbf{L}_1^{s-1}[\bm{\theta}_{[s-1,T]}] ~ \mathbf{y}_1 + \hdots + \mathbf{L}_{s-1}^{s-1}[\bm{\theta}_{[s-1,T]}] ~ \mathbf{y}_{s-1}.
$$
Given that $\mathbf{v}_t =\mathbf{y}_t - \hat{\mathbf{y}_t}  $ \eqref{out2pure}
follows and therefore the induction is completed proving (ii).

\section{Numerical Illustration, Multiperiod Portfolio Selection }

In this section we apply our design methodology to the  multiperiod 
portfolio selection problem.
There are $n$ types of assets and $N$ investment periods. The market conditions are reflected in two states representing a recessive and
expansive period respectively.
We assume  that the market conditions fluctuate according to a Markov chain $\bm{\theta} = (\theta_0, \theta_1, \hdots, \theta_{N-1}) \in \mathbb{O}^N$,
where $ \mathbb{O} = \{\text{b},\text{g}\}$, ``$\text{ b} $'' stands for the recessive period  and ``$\text{g}$ ''stands for expansive period.  The statistical description of the Markov chain is given via 
$\pi$  and transition probability matrix $ P$.
The variables of interest are 
\begin{itemize}
	\item $ \mathbf{x}_t \in \mathbb{R}^n$, where
	$ [ \mathbf{x}_t]_i$ is the position (in dollars) at time instant $t$ in the $i$-th asset, $ [ \mathbf{x}_t]_i <0 $ reflects a short position,
	\item $ \mathbf{u}_t \in \mathbb{R}^n$, where $ [ \mathbf{u}_t]_i$ is the transaction amount at time instant $t$ in the $i$-th asset, $ [ \mathbf{u}_t]_i <0 $ reflects selling,
	\item $ \mathbf{r} : \mathbb{O} \rightarrow \mathbb{R}^n$ the vector valued
	map of actual asset returns,
	\item$ \bar{\mathbf{r}} : \mathbb{O} \rightarrow \mathbb{R}^n$ the vector valued map of baseline asset returns.
\end{itemize}
The decision maker knows $ \bar{\mathbf{r}} $ and at each instant of time $t$
acquires  $ \mathbf{x}_t, \theta_t$. 
The model includes also the vectors $ \bm{\alpha},  \bm{\gamma} \in \mathbb{R}^n_+$.
The vector $ \bm{\alpha}$ bounds the allowable transaction size at each instant $t$, 
\begin{equation}
\label{bound_transaction}
| [\mathbf{u}_t]_i | \leq [\bm{\alpha}]_i, ~~ i =1, 2, \hdots, n,  ~t= 0,1,  \hdots, N-1.
\end{equation}
The vector $ \bm{\gamma}$ reflects the partial knowledge of the decision maker with regards to 
the uncertain actual asset returns in the sense that 
\begin{equation}
\label{bound_returns}
| [ \bar{\mathbf{r}}(\theta)]_i -  [ \mathbf{r}(\theta)]_i  |  \leq [\bm{\gamma}]_i, ~  i = 1, 2, \hdots, n, ~\theta \in \mathbb{O}.
\end{equation}
The asset dynamics are given by 
\begin{equation}
\label{asset_dynamics}
[ \mathbf{x}_{t+1}]_i = ( 1 +  [ \mathbf{r}(\theta_t)]_i ) ~  ( [\mathbf{x}_{t}]_i  + [\mathbf{u}_t ]_i ), ~~~ i = 1, 2, \hdots, n, ~~~ t= 0,1,  \hdots, N-1.
\end{equation}
The goals are to generate a total income stream  $ \sum_{t=0}^{N-1} \langle \mathbf{1}_n, \mathbf{u}_t \rangle$ 
close to a desired amount $U_{\text{tar}}$ and
to keep at each instant $t$ the asset holdings $x_t$ close to  the desired allocation $ \mathbf{x}_{\text{tar}} \in \mathbb{R}^n$.
We assume that we start out with a balanced portfolio in the sense that $ \mathbf{x}_0 = \mathbf{x}_{\text{tar}}$. 
The dynamics above do not fall  in the system class \eqref{system}, as is, since the actual returns which enter in  the system matrices are uncertain. 
One can use the bounds $  \bm{\alpha},  \bm{\gamma}$
in conjunction with  \eqref{bound_transaction} and \eqref{bound_returns}
in order to construct an uncertainty set for the disturbance signal.
First let 
$$
M_i  =   \max_{\theta \in \mathbb{O}} (1 + [r(\theta)]_i), ~~~ i = 1, 2, \hdots, n.
$$
Consider an asset $i \in \{1, 2, \hdots, n \}$ at any time instant $t$ one has by virtue of
\eqref{bound_transaction} and \eqref{asset_dynamics}  that
\begin{equation}
\label{bound_state}
| [ \mathbf{x}_{t+1}]_i | \leq M_i^{t+1} ~  | [ \mathbf{x}_{\text{tar}}]_i | ~+  ~ \sum_{k=1}^{t+1}
M_i^k ~ [\bm{\alpha}]_i = M_i^{t+1} ~  | [ \mathbf{x}_{\text{tar}}]_i |  ~+  ~  
\frac{M_i^{t+2}-M_i}{M_i-1}~ [\bm{\alpha}]_i.
\end{equation}
We write \eqref{asset_dynamics} as
\begin{equation*}
[ \mathbf{x}_{t+1}]_i = ( 1 +  [\bar{ \mathbf{r}}(\theta_t)]_i ) ~  ( [\mathbf{x}_{t}]_i  + [\mathbf{u}_t ]_i ) ~+~ [\mathbf{d}_{t}],  ~~~ i = 1, 2, \hdots, n, ~~~ t= 0,1,  \hdots, N-1,
\end{equation*}
with the disturbance signal given by
\begin{eqnarray*}
	[ \mathbf{d}_t ]_i    &  =  & ( -[ \bar{\mathbf{r}}(\theta_t)]_i +
	[ \mathbf{r}(\theta_t)]_i ) ~  ( [\mathbf{x}_{t}]_i  + [\mathbf{u}_t ]_i ),~~~ t= 0,1,  \hdots, N-1.
\end{eqnarray*}
Given the bound on the magnitude of the state \eqref{bound_state} and in conjuction with \eqref{bound_returns}, it follows that
\begin{eqnarray*}
	[\mathbf{d}_{t}]_i^2 &~\leq~& [\bm{\gamma}]_i^2 ~ \bigg( M_i^{t} ~   | [ \mathbf{x}_{\text{tar}}]_i | ~+~  \frac{M_i^{t+1}-1}{M_i-1}  ~  [\bm{\alpha}]_i +   [\bm{\alpha}]_i  \bigg)^2
	~=~ [\mathbf{q}_{t}]_i, ~~~ t= 0,1,  \hdots, N-1.
\end{eqnarray*}
The above inequality offers the description of the allowable disturbance signals as intersection of $n  N$ 
ellipsoids centered around the origin. 
In particular 
\begin{equation}
\label{example_uncertainty}
\mathcal{D}^N ~=~ \{~ \zeta_d ~~ |~~  [\mathbf{d}_{t}]_i^2  ~ \frac{1}{[\mathbf{q}_{t}]_i} ~ \leq ~ 1, ~~ i = 1, 2, \hdots n, ~ t = 0,1, \hdots, N-1 ~\}.
\end{equation}
In this example the decision maker observes at each instant $t$ directly the  state of the system $ \mathbf{x}_t$
and therefore the initial condition does not enter the description of the uncertainty set. 
The asset dynamics can be written now in state space form as
\begin{eqnarray}
\label{example_system}
\nonumber
\mathbf{x}_0 & = & \mathbf{z} = \mathbf{x}_{\text{tar}}, \\
\nonumber
\mathbf{x}_{t+1} & = &  \mathbf{A}[\theta_t] ~ \mathbf{x}_t + \mathbf{B}[\theta_t] ~ \mathbf{u}_t + 
\mathbf{I}_n ~ \mathbf{d}_t,  \\
\mathbf{y}_t  & = &\mathbf{I}_n ~ \mathbf{x}_t + \mathbf{O}_n ~ \mathbf{d}_t, 
\end{eqnarray}
where 
$$
\mathbf{A}[\theta_t] = \mathbf{B}[\theta_t] = \AR{cccc}{   1 +  [ \bar{\mathbf{r}}(\theta_t)]_1  & 0& \hdots & 0\\ 
	0	&  1 +  [ \bar{\mathbf{r}}(\theta_t)]_2 & \ddots& \vdots\\
	\vdots & \ddots & \ddots & 0\\ 
	0& \hdots & 0&  1 +  [ \bar{\mathbf{r}}(\theta_t)]_n}, ~~t =0,1, \hdots N-1.
$$
Given the asset dynamics \eqref{example_system} and the uncertainty description for the disturbances 
\eqref{example_uncertainty} the specifications to be satisfied robustly, i.e. for all $\zeta_d \in  \mathcal{D}^N$ are
\begin{eqnarray}
\label{example_income_deviation}
\mathbb{E}[(U_{\text{tar}} -\sum_{t=0}^{N-1} \langle \mathbf{1}_n, \mathbf{u}_t \rangle)^2  ] 
& \leq &\rho, \\
\label{example_portfolio deviation}
\mathbb{E}[ \|\mathbf{x}_t-\mathbf{x}_{\text{tar}} \|^2  ] & \leq & \mu_t,~~~ t = 1, \hdots, N.
\end{eqnarray}
In \eqref{example_income_deviation} the parameter $ \rho > 0 $ controls how much 
the cummulative income stream deviates from the aspired level $ U_{\text{tar}}$.
In \eqref{example_portfolio deviation} the positive parameters 
$ \mu_t $, $t=1,2, \hdots, N$,  control how much 
the portfolio drifts away from the desired one $ \mathbf{x}_{\text{tar}}$.
For purposes of illustration let us take $n = 2$, $N  = 3$. We consider the situation
where there is a risk-free asset, say asset 1, whose return  is not
affected by the market conditions, while the second's asset return is varying.  The data of the problem are as follows:

\begin{itemize}
	\item Portfolio optimization, 2 assets , 3 stages, 2 market states
	\begin{eqnarray*}
		\bar{\mathbf{r}}(\text{b}) & =   &  \AR{c}{ 0.02 \\  0.01 } , ~~
		\bar{\mathbf{r}}(\text{g})  =     \AR{c}{ 0.02 \\  0.05 },  ~~  
		\bm{\gamma} = \AR{c}{ 0.001 \\ 0.005 },\\
		\mathbf{x}_{\text{tar}} & = & \AR{c}{ 100 \\ 100}, ~~ 
		\bm{\alpha} = \AR{c}{ 10 \\ 10}, \\
		\pi & = & \AR{c}{ 0.1 \\ 0.9 },   P  =  \AR{cc}{ 0.2 & 0.3 \\
			0.8 &  0.7}
	\end{eqnarray*}
	\item Constraints 
	\begin{eqnarray*}
		\mathbb{E}[(U_{\text{tar}} -\sum_{t=0}^{2} \langle \mathbf{1}_2, \mathbf{u}_t \rangle)^2  ] 
		& \leq & 0.3, \\
		\mathbb{E}[ \|\mathbf{x}_1-\mathbf{x}_{\text{tar}} \|^2  ] & \leq & 5, \\
		\mathbb{E}[ \|\mathbf{x}_2-\mathbf{x}_{\text{tar}} \|^2  ] & \leq & 10, \\
		\mathbb{E}[ \|\mathbf{x}_3-\mathbf{x}_{\text{tar}} \|^2  ] & \leq & 20.
	\end{eqnarray*}
\end{itemize}
We set the desired income to $ U_{\text{tar}} = -17.675$
computed via ``naive'' re-balancing procedure using the baseline returns.  In particular naive re-balancing  amounts to choosing
$$
[\mathbf{u}_t]_i =     - \frac{[\bar{\mathbf{r}}(\theta_t)]_i}{1 + [\bar{\mathbf{r}}(\theta_t)]_i },  ~~~ i = 1, 2, \hdots, n, ~~~ t= 0,1,  \hdots, N-1,
$$
so that $ \mathbf{x}_t =  \mathbf{x}_{\text{tar}}$, under the baseline 
returns.

We run the simulation in a MATLAB environment on 
a MacBook Pro with a 2.3 GHz Intel Core i5 Processor, utilizing
8 GB of RAM memory. The semidefinite program was formulated 
and solved with CVX  (Grant and Boyd 2014, Grant et al. 2006)
using the SDPT3, Toh et al. (1999), solver. We set the switching memory to $ T = 2 $ and  consider  100 realizations of the uncertain data.

\begin{itemize}
	\item State trajectory
	\begin{figure}[!h]
		\centering
		\includegraphics[width=80mm]{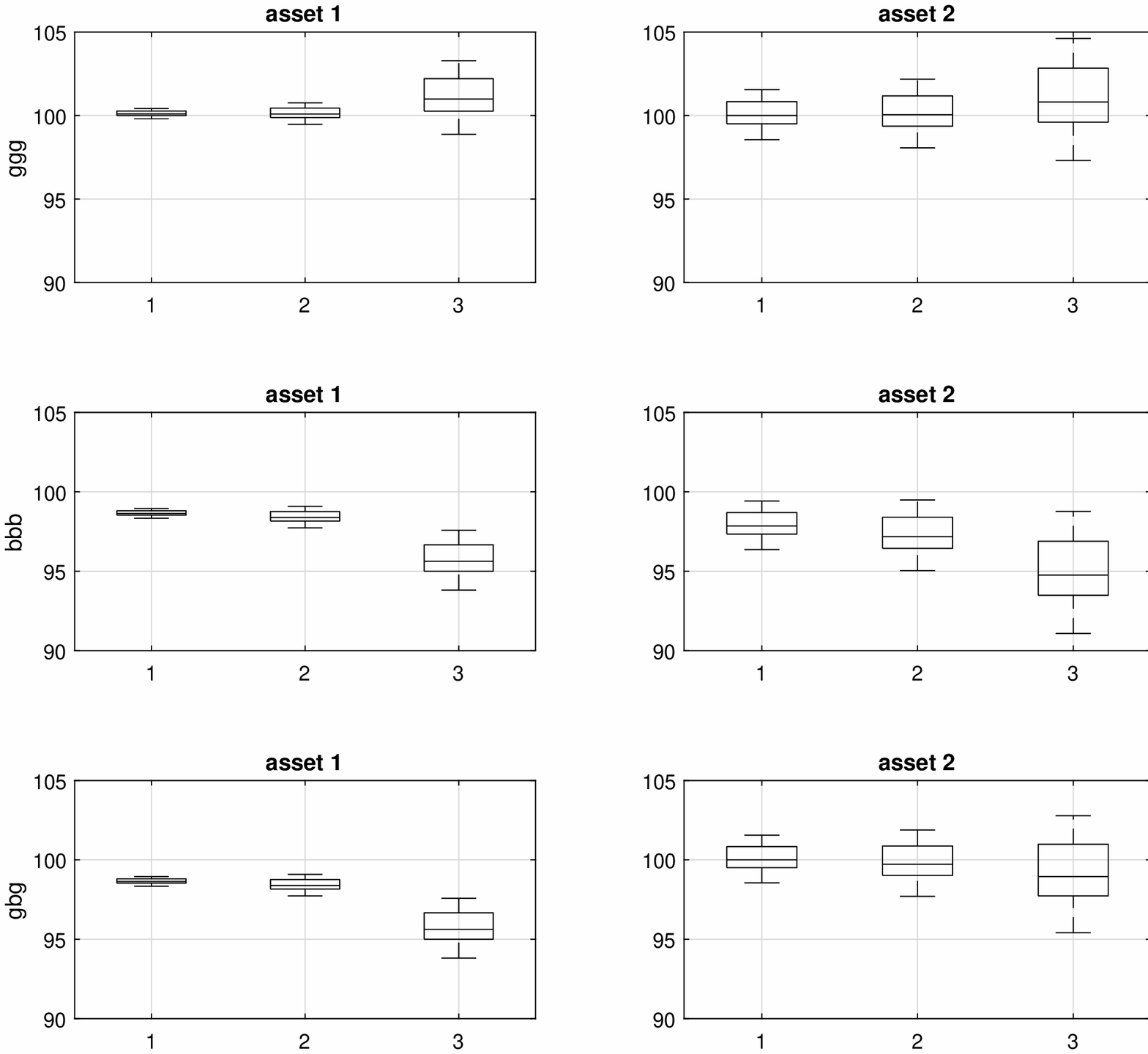}
	\end{figure} 
\end{itemize}
The above figure depicts boxplots of the state trajectory for 3 different realizations of the regime-switching signal. As  time progresses the range of the realized state variables becomes progressively larger given the compounding effect of the uncertainty.
The achieved performance tends to be better for the more
probable sequences depicted.

The 8 possible scenarios, corresponding to the realizations of the regime-switching signal are categorized in matrix form as
$$
\AR{cccccccc}{ \text{b} & \text{b} & \text{b} & \text{b} & \text{g} & \text{g} & \text{g} & \text{g}  \\ 
	\text{b} & \text{b} & \text{g} & \text{g} & \text{b} & \text{b} & \text{g} & \text{g}  \\ 
	\text{b} & \text{g} & \text{b} & \text{g} & \text{b} & \text{g} & \text{b} & \text{g}  }
$$
The next figure depicts  the average over the 100 samples of 
$$| U_{\text{tar}} - \AR{cc}{1 & 1} ~ (\sum_{t=0}^2 \mathbf{u}_t) |^2$$
for every possible realization of the underlying regime-switching signal. 
\begin{figure}[!h]
	\centering
	\includegraphics[width=60mm]{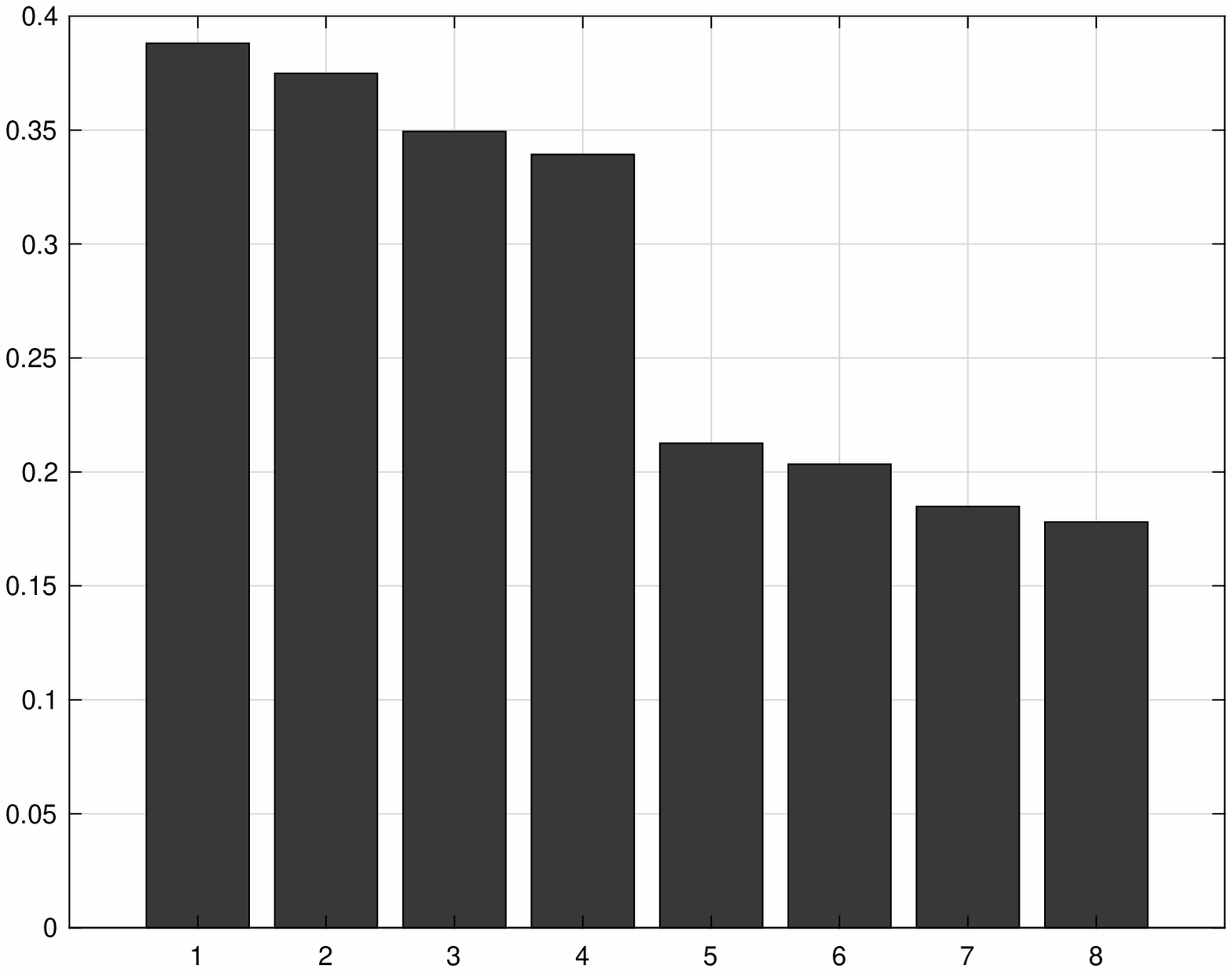}
\end{figure} 
If we start out on a recession  as far as the initial conditions are concerned, first four columns, then the quadratic constraint on the income stream is violated. The probability of this event though is  $ 0.1$
and the quadratic constraint on the deviation from the desired income stream is satisfied in expectation, as required by our method.

\section{Summary and Outlook}

In this paper we provide a computationally tractable procedure for designing affine policies 
for a constrained  robust multistage decision problem as it pertains to linear stochastic systems subject to Markovian-switching.  The main contribution is the fact that we were able to  point out an explicit semidefinite feasibility problem whose solution yields an affine purified-output-based policy meeting the design specifications robustly w.r.t. the uncertain data. Moreover for fixed memory of the switching signal, the sizes of the feasibility problem and the complexity of solving it are polynomial in the sizes of the data and the finite horizon under consideration.

In terms of future work, we will consider the case when 
the switching signal is not observable by the decision maker,
and when the statistics of the switching process are uncertain themselves.


%

\section{Appendices}

\subsection*{A1}

This section contains the block matrices that connect the initial state, noise and disturbance trajectories to the state, control 
and purified output trajectories.
For $t, \tau \in \mathbb{Z}_{[0,N]}$, $ \bm{\theta} \in \mathbb{O}^N$ denote the  state transition matrix by 
$$
\bm{\Gamma}[t,\tau] = \begin{cases}
 I     ~~&\text{if} ~~ t = \tau, \\
 A_{t-1}(\bm{\theta}_{t-1}) \hdots A_{\tau}(\bm{\theta}_{\tau}) ~~&\text{if} ~~ t > \tau.
\end{cases}
$$
$\underline{\mathbf{A}} \in \mathbb{R}^{n_x  N \times n_x},$ 	
	is a lower triangular block matrix, consisting of $N$ blocks of size $ n_x \times n_x$.
	\begin{equation*}
	\label{z2x}
	\underline{\mathbf{A}} = \AR{c} { \bm{\Gamma}[1,0]  \\ \bm{\Gamma}[2,0] \\ \vdots \\ \bm{\Gamma}[N,0]}. 
	\end{equation*}
$\underline{\mathbf{B}} \in \mathbb{R}^{n_x  N \times n_u N},$ 	
is a  block matrix, consisting of $N^2$ blocks of size $ n_x \times n_u$.
The $(k,j)$-th block is
$$
[\underline{\mathbf{B}}]_{k,j} = 	\begin{cases}
\bm{\Gamma}[k,j+1] ~ B_j[\bm{\theta}_j], ~& k \geq j, \\
~~~~~0_{n_x \times n_u}, ~& k < j,~~ k,j \in \mathbb{Z}_{[0,N-1]}.
\end{cases}
$$	
$\underline{\mathbf{B}}^{(d)}  \in \mathbb{R}^{n_x  N \times n_d N},$ 	
is a lower triangular block matrix, consisting of $N^2$ blocks of size $ n_x \times n_d$.
The $(k,j)$-th block is
$$
[\underline{\mathbf{B}}^{(d)}]_{k,j} = 	\begin{cases}
\bm{\Gamma}[k,j+1] ~ B_j^{(d)}[\bm{\theta}_j], ~& k \geq j, \\
~~~~~0_{n_x \times n_d}, ~& k < j,~~ k,j \in \mathbb{Z}_{[0,N-1]}.
\end{cases}
$$	
$\underline{\mathbf{B}}^{(s)}  \in \mathbb{R}^{n_x  N \times n_e N},$ 	
is a lower triangular block matrix, consisting of $N^2$ blocks of size $ n_x \times n_e$.
The $(k,j)$-th block is
$$
[\underline{\mathbf{B}}^{(s)}]_{k,j} = 	\begin{cases}
\bm{\Gamma}[k,j+1] ~ B_j^{(s)}[\bm{\theta}_j], ~& k \geq j, \\
~~~~~0_{n_x \times n_e}, ~& k < j,~~ k,j \in \mathbb{Z}_{[0,N-1]}.
\end{cases}
$$	
$  \underline{\mathbf{C}}  \in \mathbb{R}^{n_y N \times n_x},$ 
	is a block matrix, consisting of $N$ blocks of size $ n_y \times n_x$.
	$$
	\underline{\mathbf{C}} = \AR{c} { C_0[\bm{\theta}_0]  \\ C_1[\bm{\theta}_1]  \bm{\Gamma}[1,0]  \\ \vdots \\ C_{N-1}[\bm{\theta}_{N-1}] ~ \bm{\Gamma}[N-1,0] }.
$$
$\underline{\mathbf{D}}^{(d)}  \in \mathbb{R}^{n_y  N \times n_d N},$ 	
is a lower triangular block matrix, consisting of $N^2$ blocks of size $ n_y \times n_d$.
The $(k,j)$-th block is
$$
[\underline{\mathbf{D}}^{(d)}]_{k,j} = 	\begin{cases}
D_j^{(d)}[\bm{\theta}_j]~& k = j, \\
C_k[\bm{\theta}_k] ~  \bm{\Gamma}[k,j+1]~   B_j^{(d)}[\bm{\theta}_j], ~& k > j, \\
0_{n_y \times n_d}, ~& k < j,~ 
\end{cases}
$$	
$k,j \in \mathbb{Z}_{[0,N-1]}.$\\
\noindent 
$\underline{\mathbf{D}}^{(s)}  \in \mathbb{R}^{n_y  N \times n_e N},$ 	
is a lower triangular block matrix, consisting of $N^2$ blocks of size $ n_y \times n_e$.
The $(k,j)$-th block is
$$
[\underline{\mathbf{D}}^{(s)}]_{k,j} = 	\begin{cases}
D_j^{(s)}[\bm{\theta}_j]~& k = j, \\
C_k[\bm{\theta}_k] ~  \bm{\Gamma}[k,j+1]~   B_j^{(s)}[\bm{\theta}_j], ~& k > j, \\
0_{n_y \times n_e}, ~& k < j,~ 
\end{cases}
$$	
$k,j \in \mathbb{Z}_{[0,N-1]}.$\\
\noindent The block vector $ \underline{\mathbf{h}} \in \mathbb{R}^{n_u  N \times 1}$ consist of $N$ vectors of size $ n_u$.
$$
\underline{\mathbf{h}} = \AR{c}{h_0[\bm{\theta}_{[0,T]}] \\  h_1[\bm{\theta}_{[1,T]}] \\ \vdots \\ h_{N-1}[\bm{\theta}_{[N-1,T]}]}.
$$
The  block lower triangular matrix 
$\underline{\mathbf{H}} \in \mathbb{R}^{n_u  N \times n_y  N} $ consist of $N^2$ matrices of size $ n_u \times n_y$.
The $(k,j)$-th block is
$$
[\underline{\mathbf{H}}]_{k,j} = 	\begin{cases}
H_j^k[\bm{\theta}_{[k,T]}], ~& k \geq j, \\
~~~~~0_{n_u \times n_y}, ~& k < j,~~ k,j \in \mathbb{Z}_{[0,N-1]}.
\end{cases}
$$

\subsection*{A2}

This section contains the approximate version of the inhomogeneous 
S-lemma, see also \cite{bental6}.   An improved tightness factor 
is adapted from \cite{jud_17}.
Consider the following quadratically constrained optimization problem:
$$
\text{Opt} = \max_{\mathbf{x}} \{ ~ f(\mathbf{x}) = \mathbf{x}^\T ~ \mathbf{A} ~\mathbf{x} ~+~ 2 ~ \langle \mathbf{b}, \mathbf{x} \rangle~:~  \mathbf{x}^\T ~ \mathbf{A}_i ~\mathbf{x} ~ \leq ~ c_i, ~ i = 1, \hdots, m ~ \},
$$
where $c_i > 0$, $\mathbf{A}_i \succeq 0 $, $i = 1, \hdots, m$. The matrix $\mathbf{A}$ is symmetric and $ \sum_{i=1}^m \mathbf{A}_i \succ \mathbf{O}$.
Let SDP be the optimal value in the semidefinite relaxation of this problem:
$$
\text{SDP} = \min_{\omega, \bm{\lambda}} \{ ~ \omega ~:~ 
\AR{ccc}{\omega - \sum_{i=1}^m c_i ~ [\bm{\lambda}]_i &  & - \mathbf{b}^\T \\ &  &  \\ 
	- \mathbf{b} &  & \sum_{i=1}^m  [\bm{\lambda}]_i ~ \mathbf{A}_i ~ - ~ \mathbf{A}} \succeq \mathbf{O}, ~~~\bm{\lambda} \geq \mathbf{0}_m. ~\}
$$
Then 
\begin{equation}
\label{sandwich}
\text{Opt} ~ \leq ~ \text{SDP} ~\leq ~ \Theta ~ \text{Opt}
\end{equation}
where $ \Theta = 1 $ in the case of $m= 1$ and 
\begin{equation}
\label{factor}
\Theta = 2~  \ln \bigg(6 \sum_{i=1}^m \rank(\mathbf{A}_i) \bigg)
\end{equation}
in the case of $m > 1$. 
The factor $\Theta $ was improved in \cite{jud_17}, proposition 3.3
to 
$$
\Theta = 2~  \ln(m) +  2 \sqrt{\ln(m) } +1
$$
for the case of a homogeneous quadratic form and consequently to 
\begin{equation}
\label{factor_improved}
\Theta = 2~  \ln(m+1) +  2 \sqrt{\ln(m+1) } +1
\end{equation}
for the case under consideration.

\subsection*{A3}

Proof of Proposition~\ref{lemma1}.
The proof will utilize repeatedly the law of iterated expectations and the Markov property. 
First observe that the matrix valued functions defined in \eqref{recursive_linear_1} to
\eqref{recursive_linear_4},
can be written as sequential conditional expectations as follows:
\begin{eqnarray*}
	\nonumber
	\mathbf{g}_{N-2}[\theta_{N-2}] & = & \mathbb{E}[\mathbf{f}_{N-1}[\theta_{N-1}] ~|~\theta_{N-2}],  \\
	\nonumber
	\mathbf{g}_{s}[\theta_{s}] & = &   \mathbb{E}[\mathbf{g}_{s+1}[\theta_{s+1}]~ \mathbf{f}_{s+1}[\theta_{s+1}] ~|~\theta_s], ~~~
	s = N-3, N-4, \hdots, \tau+1, \tau,  \\
	\mathbf{g}_{\tau-T-1}[\theta_{\tau-T-1}] & = &
	\mathbb{E}[\mathbf{g}_{\tau}[\theta_{\tau}] ~ \mathbf{f}_{\tau}[\bm{\theta}_{[\tau,T]}] ~ 
	\mathbf{f}_{\tau-1}[\theta_{\tau-1}]  ~ \hdots ~ \mathbf{f}_{\tau-T}[\theta_{\tau-T}]~|~\theta_{\tau-T-1}], \\
	\mathbf{g}_{s}[\theta_{s}] & = &  \mathbb{E}[\mathbf{g}_{s+1}[\theta_{s+1}]~ \mathbf{f}_{s+1}[\theta_{s+1}] ~|~\theta_s], ~~~
	s = \tau - T -2 ,  \tau - T - 3, \hdots, 0.  
\end{eqnarray*}
Using the above expresssions and by repeatedly applying the law of iterated expectations, gives
\begin{eqnarray*}
	\mathbf{T}_l & = & \mathbb{E}\bigg[~   \mathbb{E}[\mathbf{f}_{N-1}[\theta_{N-1}] ~|~\theta_{N-2}]~
	\mathbf{F}_{N-2, \tau+1}[\bm{\theta}_{[N-2,N-\tau-3]}]~
	\mathbf{f}_{\tau}[\bm{\theta}_{[\tau,T]}] ~ 
	\mathbf{F}_{\tau-1, 0}[\bm{\theta}_{[\tau-1,\tau-1]}]
	\bigg] \\
	& = & 
	\mathbb{E}\bigg[\mathbf{g}_{N-2}[\theta_{N-2}]~  
	\mathbf{F}_{N-2, \tau+1}[\bm{\theta}_{[N-2,N-\tau-3]}]~
	\mathbf{f}_{\tau}[\bm{\theta}_{[\tau,T]}] ~ 
	\mathbf{F}_{\tau-1, 0}[\bm{\theta}_{[\tau-1,\tau-1]}] \bigg] \\
	& = & 
	\mathbb{E}\bigg[\mathbb{E}[\mathbf{g}_{N-2}[\theta_{N-2}] ~\mathbf{f}_{N-2}[\theta_{N-2}] ~ |~\theta_{N-3}]~   
	\mathbf{F}_{N-3, \tau+1}[\bm{\theta}_{[N-3,N-\tau-4]}]~
	\mathbf{f}_{\tau}[\bm{\theta}_{[\tau,T]}] ~ 
	\mathbf{F}_{\tau-1, 0}[\bm{\theta}_{[\tau-1,\tau-1]}] \bigg] \\
	&=&\mathbb{E}\bigg[\mathbf{g}_{N-3}[\theta_{N-3}]~  
	\mathbf{F}_{N-3, \tau+1}[\bm{\theta}_{[N-3,N-\tau-4]}]~
	\mathbf{f}_{\tau}[\bm{\theta}_{[\tau,T]}] ~ 
	\mathbf{F}_{\tau-1, 0}[\bm{\theta}_{[\tau-1,\tau-1]}]\bigg] \\
	& \vdots & \\
	& = & \mathbb{E}\bigg[\mathbf{g}_{\tau}[\theta_{\tau}]~  
	\mathbf{f}_{\tau}[\bm{\theta}_{[\tau,T]}] ~ 
	\mathbf{F}_{\tau-1, 0}[\bm{\theta}_{[\tau-1,\tau-1]}]\bigg] \\
	& = & \mathbb{E}\bigg[\mathbb{E}[\mathbf{g}_{\tau}[\theta_{\tau}] ~
	\mathbf{f}_{\tau}[\bm{\theta}_{[\tau,T]}] ~ 
	\mathbf{F}_{\tau-1, \tau-T}[\bm{\theta}_{[\tau-1,T-1]}]  ~|~\theta_{\tau-T-1}]~   
	\mathbf{F}_{\tau-T-1, 0}[\bm{\theta}_{[\tau-T-1,\tau-T-1]}] \bigg] \\
	&=& \mathbb{E}\bigg[\mathbf{g}_{\tau-T-1}[\theta_{\tau-T-1}]~  
	\mathbf{F}_{\tau-T-1, 0}[\bm{\theta}_{[\tau-T-1,\tau-T-1]}] \bigg] \\
	& \vdots & \\
	& = &  \mathbb{E}\bigg[\mathbf{g}_{0}[\theta_{0}]~  
	\mathbf{f}_0[\theta_0]\bigg] 
	= \sum_{\theta_0 \in \mathbb{O}}	\Pr[ \theta_0 ] ~ \mathbf{g}_0[\theta_0] ~\mathbf{f}_0[\theta_0].
\end{eqnarray*}

\vspace{0.1in}

\subsection*{A4}

This section contains the proof of Proposition~\ref{lemma2}
as well as  the terms appearing 
in $ \mathcal{V}_j[ \chi] $.

\begin{eqnarray*}
	\mathcal{V}_{(I,I)} & = & \mathbb{E}\bigg[\underline{\mathbf{h}}^\T ~ \underline{\mathbf{B}}^\T ~ \mathcal{A}_{xx,j}~ 
	\underline{\mathbf{B}}~ \underline{\mathbf{h}}\bigg] ~+~ 2 ~ \mathbb{E}\bigg[ \underline{\mathbf{h}}^\T ~ \underline{\mathbf{B}}^\T ~ 
	\mathcal{A}_{xu,j}~  \underline{\mathbf{h}}\bigg] 
	~ + ~ \mathbb{E}\bigg[ \underline{\mathbf{h}}^\T ~ \mathcal{A}_{uu,j} ~  \underline{\mathbf{h}} \bigg]\\
	\mathcal{V}_{(I,II)}  & = & \mathbb{E}\bigg[\underline{\mathbf{h}}^\T ~ \underline{\mathbf{B}}^\T ~ \mathcal{A}_{xx,j}~ 
	(	\underline{\mathbf{B}}~ \underline{\mathbf{H}}~ \underline{\mathbf{C}}~ + ~ \underline{\mathbf{A}} )\bigg] ~+~
	\mathbb{E}\bigg[\underline{\mathbf{h}}^\T  ~ \mathcal{A}_{ux,j} ~ 
	(	\underline{\mathbf{B}}~ \underline{\mathbf{H}}~ \underline{\mathbf{C}}~ + ~ \underline{\mathbf{A}} ) \bigg] 
	~+ ~  \mathbb{E}\bigg[\underline{\mathbf{h}}^\T ~\underline{\mathbf{B}}^\T~ \mathcal{A}_{xu,j} ~ \underline{\mathbf{H}}~ \underline{\mathbf{C}}\bigg] \\ ~&+&~
	\mathbb{E}\bigg[ \underline{\mathbf{h}}^\T ~ \mathcal{A}_{uu,j}~\underline{\mathbf{H}}~ \underline{\mathbf{C}}   \bigg]	\\
	\mathcal{V}_{(I,III)}  & = & \mathbb{E}\bigg[\underline{\mathbf{h}}^\T ~ \underline{\mathbf{B}}^\T ~ \mathcal{A}_{xx,j}~ 
	(	\underline{\mathbf{B}}~ \underline{\mathbf{H}}~ \underline{\mathbf{D}}~ + ~ \underline{\mathbf{R}} )\bigg] ~+~ \mathbb{E}\bigg[
	\underline{\mathbf{h}}^\T  ~ \mathcal{A}_{ux,j} ~ 
	(	\underline{\mathbf{B}}~ \underline{\mathbf{H}}~ \underline{\mathbf{D}}~ + ~ \underline{\mathbf{R}} ) \bigg] 
	~+ ~ \mathbb{E}\bigg[ \underline{\mathbf{h}}^\T ~\underline{\mathbf{B}}^\T~ \mathcal{A}_{xu,j} ~ \underline{\mathbf{H}}~ \underline{\mathbf{D}}\bigg]\\ ~&+&~
	\mathbb{E}\bigg[\underline{\mathbf{h}}^\T ~ \mathcal{A}_{uu,j}~\underline{\mathbf{H}}~ \underline{\mathbf{D}}   \bigg]	\\
	\mathcal{V}_{(II,I)}  & = & 	\mathcal{V}_{(I,II)} \\
	\mathcal{V}_{(II,II)}  & = &   \mathbb{E}\bigg[
	(	\underline{\mathbf{C}}^\T~ \underline{\mathbf{H}}^\T~ \underline{\mathbf{B}}^\T~ + ~ \underline{\mathbf{A}}^\T ) ~ \mathcal{A}_{xx,j}~ (	\underline{\mathbf{B}}~ \underline{\mathbf{H}}~ \underline{\mathbf{C}}~ + ~ \underline{\mathbf{A}} ) \bigg]
	~+~ 2 ~ 	\mathbb{E}\bigg[(	\underline{\mathbf{C}}^\T~ \underline{\mathbf{H}}^\T~ \underline{\mathbf{B}}^\T~ + ~ \underline{\mathbf{A}}^\T ) ~ 
	\mathcal{A}_{xu,j}~\underline{\mathbf{H}}~ \underline{\mathbf{C}}\bigg] \\
	~  &+& ~ 	\mathbb{E}\bigg[\underline{\mathbf{C}}^\T~ \underline{\mathbf{H}}^\T ~ \mathcal{A}_{uu,j} ~  
	\underline{\mathbf{H}}~ \underline{\mathbf{C}} \bigg] \\
	\mathcal{V}_{(II,III)}  & = &  
	\mathbb{E}\bigg[(	\underline{\mathbf{C}}^\T~ \underline{\mathbf{H}}^\T~ \underline{\mathbf{B}}^\T~ + ~ \underline{\mathbf{A}}^\T )  ~ \mathcal{A}_{xx,j}~ 
	(	\underline{\mathbf{B}}~ \underline{\mathbf{H}}~ \underline{\mathbf{D}}~ + ~ \underline{\mathbf{R}} )\bigg] ~+~
	\mathbb{E}\bigg[\underline{\mathbf{C}}^\T~ \underline{\mathbf{H}}^\T ~ \mathcal{A}_{ux,j} ~ 
	(	\underline{\mathbf{B}}~ \underline{\mathbf{H}}~ \underline{\mathbf{D}}~ + ~ \underline{\mathbf{R}} ) \bigg]\\
	~&+& ~ \mathbb{E}\bigg[ (	\underline{\mathbf{C}}^\T~ \underline{\mathbf{H}}^\T~ \underline{\mathbf{B}}^\T~ + ~ \underline{\mathbf{A}}^\T ) ~ \mathcal{A}_{xu,j} ~ \underline{\mathbf{H}}~ \underline{\mathbf{D}}\bigg] ~ +~
	\mathbb{E}\bigg[\underline{\mathbf{C}}^\T~ \underline{\mathbf{H}}^\T ~ \mathcal{A}_{uu,j}~\underline{\mathbf{H}}~ \underline{\mathbf{D}}   \bigg]	\\
	\mathcal{V}_{(III,I)}  & = & 	\mathcal{V}_{(I,III)}  \\
	\mathcal{V}_{(III,II)}  & = & 	\mathcal{V}_{(II,III)}  \\
	\mathcal{V}_{(III,III)}  & = &   \mathbb{E}\bigg[
	(	\underline{\mathbf{D}}^\T~ \underline{\mathbf{H}}^\T~ \underline{\mathbf{B}}^\T~ + ~ \underline{\mathbf{R}}^\T ) ~ \mathcal{A}_{xx,j}~ (	\underline{\mathbf{B}}~ \underline{\mathbf{H}}~ \underline{\mathbf{D}}~ + ~ \underline{\mathbf{R}} ) \bigg]
	~+~ 2 ~ 	\mathbb{E}\bigg[(	\underline{\mathbf{D}}^\T~ \underline{\mathbf{H}}^\T~ \underline{\mathbf{B}}^\T~ + ~ \underline{\mathbf{R}}^\T ) ~ 
	\mathcal{A}_{xu,j}~\underline{\mathbf{H}}~ \underline{\mathbf{D}}\bigg] \\
	~ & + &~ 	\mathbb{E}\bigg[\underline{\mathbf{D}}^\T~ \underline{\mathbf{H}}^\T ~ \mathcal{A}_{uu,j} ~  
	\underline{\mathbf{H}}~ \underline{\mathbf{D}} \bigg]	  		 
\end{eqnarray*}

Proof of Proposition~\ref{lemma2}.	
The proof will utilize repeatedly the law of iterated expectations and the Markov property. 
First observe that the matrix valued functions defined in \eqref{recursive_quadratic}
can be written as sequential conditional expectations as follows:
\begin{eqnarray*}
	\nonumber
	\mathbf{g}_{N-2}[\theta_{N-2}] & = & \mathbb{E}[\mathbf{q}_{N-1}[\theta_{N-1}]^\T~ \mathbf{S} ~ \mathbf{f}_{N-1}[\theta_{N-1}] ~|~\theta_{N-2}],  \\
	\nonumber
	\mathbf{g}_{s}[\theta_{s}] & = &   \mathbb{E}[\mathbf{q}_{s+1}[\theta_{s+1}]^\T ~ \mathbf{g}_{s+1}[\theta_{s+1}]~ \mathbf{f}_{s+1}[\theta_{s+1}] ~|~\theta_s], ~~~
	s = N-3, N-4, \hdots, \tau+1, \tau,  \\
	\mathbf{g}_{\tau-T-1}[\theta_{\tau-T-1}] & = &
	\mathbb{E}[\bigg( \mathbf{q}_{\tau}[\bm{\theta}_{[\tau,T]}] ~ 
	\mathbf{Q}_{\tau-1, \tau-T}[\bm{\theta}_{[\tau-1,T-1]}] \bigg)^\T ~ \mathbf{g}_{\tau}[\theta_{\tau}] \\
	&& 
	\bigg( \mathbf{f}_{\tau}[\bm{\theta}_{[\tau,T]}] ~ 
	\mathbf{F}_{\tau-1, \tau-T}[\bm{\theta}_{[\tau-1,T-1]}] \bigg)~|~\theta_{\tau-T-1}], \\
	\mathbf{g}_{s}[\theta_{s}] & = & \mathbb{E}[\mathbf{q}_{s+1}[\theta_{s+1}]^\T ~ \mathbf{g}_{s+1}[\theta_{s+1}]~ \mathbf{f}_{s+1}[\theta_{s+1}] ~|~\theta_s], ~~~
	s = \tau - T -2 ,  \tau - T - 3, \hdots, 0.  
\end{eqnarray*}
Using the above expresssions and revisiting \eqref{typical_quadratic_term} gives
\begin{eqnarray*}
	\mathbf{T}_q & = & \mathbb{E}\bigg[\bigg( \mathbf{Q}_{N-2, \tau+1}[ \bm{\theta}_{[N-2,N-\tau-3]}  ] ~ 
	\mathbf{q}_{\tau}[\bm{\theta}_{[\tau,T]}] ~ 
	\mathbf{Q}_{\tau-1,0}[ \bm{\theta}_{[\tau-1,\tau-1]} ]
	\bigg)^\T~ 
	\mathbb{E}[\mathbf{q}_{N-1}[\theta_{N-1}]^\T~ \mathbf{S} ~ \mathbf{f}_{N-1}[\theta_{N-1}] ~|~\theta_{N-2}] ~   \\
	\nonumber
	&&
	\bigg( \mathbf{F}_{N-2,\tau+1}[ \bm{\theta}_{[N-2,N-\tau-3]} ] ~ 
	\mathbf{f}_{\tau}[\bm{\theta}_{[\tau,T]}] ~ 
	\mathbf{F}_{\tau-1,0}[\bm{\theta}_{[\tau-1,\tau-1]}]
	\bigg)\bigg] \\
	& = & 
	\mathbb{E}\bigg[\bigg( \mathbf{Q}_{N-2,\tau+1}[\bm{\theta}_{[N-2,N-\tau-3]}] ~ 
	\mathbf{q}_{\tau}[\bm{\theta}_{[\tau,T]}] ~ 
	\mathbf{Q}_{\tau-1,0}[\bm{\theta}_{[\tau-1,\tau-1]}]
	\bigg)^\T~ 
	\mathbf{g}_{N-2}[\theta_{N-2}] ~   \\
	\nonumber
	&&
	\bigg( \mathbf{F}_{N-2,\tau+1}[\bm{\theta}_{[N-2,N-\tau-3]}] ~ 
	\mathbf{f}_{\tau}[\bm{\theta}_{[\tau,T]}] ~ 
	\mathbf{F}_{\tau-1,0}[\bm{\theta}_{[\tau-1,\tau-1]}]
	\bigg)\bigg] \\
	\nonumber
	&=&
	\mathbb{E}\bigg[\bigg( \mathbf{Q}_{N-3,\tau+1}[\bm{\theta}_{[N-3,N-\tau-4]}] ~ 
	\mathbf{q}_{\tau}[\bm{\theta}_{[\tau,T]}] ~ 
	\mathbf{Q}_{\tau-1,0}[\bm{\theta}_{[\tau-1,\tau-1]}]
	\bigg)^\T~ 
	\mathbb{E}[\mathbf{q}_{N-2}[\theta_{N-2}]^\T~ \mathbf{g}_{N-2}[\theta_{N-2}] ~ \mathbf{f}_{N-2}[\theta_{N-2}] ~|~\theta_{N-3}] ~   \\
	\nonumber
	&&
	\bigg( \mathbf{F}_{N-3,\tau+1}[\bm{\theta}_{[N-3,N-\tau-4]}] ~ 
	\mathbf{f}_{\tau}[\bm{\theta}_{[\tau,T]}] ~ 
	\mathbf{F}_{\tau-1,0}[\bm{\theta}_{[\tau-1,\tau-1]}]
	\bigg)\bigg] \\
	\nonumber
	&=&   
	\mathbb{E}\bigg[\bigg( \mathbf{Q}_{N-3,\tau+1}[\bm{\theta}_{[N-3,N-\tau-4]}] ~ 
	\mathbf{q}_{\tau}[\bm{\theta}_{[\tau,T]}] ~ 
	\mathbf{Q}_{\tau-1,0}[\bm{\theta}_{[\tau-1,\tau-1]}]
	\bigg)^\T~ 
	\mathbf{g}_{N-3}[\theta_{N-3}]  ~   \\
	\nonumber
	&&
	\bigg( \mathbf{F}_{N-3,\tau+1}[\bm{\theta}_{[N-3,N-\tau-4]}] ~ 
	\mathbf{f}_{\tau}[\bm{\theta}_{[\tau,T]}] ~ 
	\mathbf{F}_{\tau-1,0}[\bm{\theta}_{[\tau-1,\tau-1]}]
	\bigg)\bigg] \\
	\nonumber 
	& \vdots & \\
	\nonumber
	&=&
	\mathbb{E}\bigg[\bigg( 
	\mathbf{q}_{\tau}[\bm{\theta}_{[\tau,T]}] ~ 
	\mathbf{Q}_{\tau-1,0}[\bm{\theta}_{[\tau-1,\tau-1]}]
	\bigg)^\T~ 
	\mathbf{g}_{\tau}[\theta_{\tau}]  ~   
	\bigg(
	\mathbf{f}_{\tau}[\bm{\theta}_{[\tau,T]}] ~ 
	\mathbf{F}_{\tau-1,0}[\bm{\theta}_{[\tau-1,\tau-1]}]
	\bigg)\bigg] \\
	\nonumber 	
	&=&\mathbb{E}\bigg[\bigg(  \mathbf{Q}_{\tau-T-1,0}[\bm{\theta}_{[\tau-T-1,\tau-T-1]}]
	\bigg)^\T \\
	\nonumber
	&& ~ \mathbb{E}[ \bigg( 
	\mathbf{q}_{\tau}[\bm{\theta}_{[\tau,T]}] ~ 
	\mathbf{Q}_{\tau-1,\tau-T}[\bm{\theta}_{[\tau-1,T-1]}]
	\bigg)^\T \mathbf{g}_{\tau}[\theta_{\tau}] ~ 	\mathbf{f}_{\tau}[\bm{\theta}_{[\tau,T]}] ~ 
	\mathbf{F}_{\tau-1,\tau-T}[\bm{\theta}_{[\tau-1,T-1]}] ~|~\theta_{\tau-T-1}] ~ \\
	\nonumber && 
	\bigg(  \mathbf{F}_{\tau-T-1,0}[\bm{\theta}_{[\tau-T-1,\tau-T-1]}]
	\bigg) \bigg] \\
	\nonumber
	& = & \mathbb{E}\bigg[\bigg(  \mathbf{Q}_{\tau-T-1,0}[\bm{\theta}_{[\tau-T-1,\tau-T-1]}]
	\bigg)^\T ~ \mathbf{g}_{\tau-T-1}[\theta_{\tau-T-1}] ~ 	\bigg(  \mathbf{F}_{\tau-T-1,0}[\bm{\theta}_{[\tau-T-1,\tau-T-1]}]
	\bigg) \bigg] \\
	& \vdots & \\
	& = &  \mathbb{E}\bigg[\mathbf{q}_0[\theta_0]^\T~  \mathbf{g}_{0}[\theta_{0}]~  \mathbf{f}_0[\theta_0]\bigg] 
	=  \sum_{\theta_0 \in \mathbb{O}}	\Pr[ \theta_0 ] ~ \mathbf{q}_0[\theta_0]^\T ~ \mathbf{g}_0[\theta_0] ~\mathbf{f}_0[\theta_0].
\end{eqnarray*}

\vspace{0.1in}

\section*{Acknowledgment}

The authors would like to thank Arkadi Nemirovski for his critical comments and suggestions. They also gratefully acknowledge the support of National Science Foundation grants 1637473 and 1637474.


\begin{thebibliography}{1}

\bibitem{fleming}
W. H. Fleming, R.W. Rishel, 
{\it Deterministic and Stochastic Optimal Control} 
Springer-Verlag, New York, NY, 1975

\bibitem{bertsekas1} 
D. P. Bertsekas, S. E. Shreve,
{\it Stochastic Optimal Control: The Discrete-Time Case} 
Academic Press, New York, NY, 1978.


\bibitem{sutton} 
R.S. Sutton, A. G. Barto,
{\it Reinforcement Learning: An Introduction} 
The MIT Press, Cambridge, MA, 1998.

\bibitem{bertsekas2} 
D. P. Bertsekas,
{\it Dynamic Programming and Optimal Control, Vol. II} 
Athena Scientific, Nashua, NH, 2012.

\bibitem{birge} 
J. R. Birge, F. Louveaux, 
{\it Introduction to Stochastic Programming}
Springer-Verlag, New York, NY, 2011.

\bibitem{shapiro} 
A. Shapiro, D. Dentcheva, A. Ruszczynski,
{\it Lectures on Stochastic Programming: Modeling and Theory}
Society for Industrial and Applied Mathematics, Philadelphia, PA, 2009.

\bibitem{blanchini}
F. Blanchini, S. Miani, 
{\it Set theoretic methods in control }
Birkh{\"a}user, Boston, MA, 2015.

\bibitem{kurzhanski}
A. B. Kurzhanski, P. Varaiya,
{\it Dynamics and Control of Trajectory Tubes}
Birkh{\"a}user, Boston, MA, 2014.

\bibitem{rawlings}
J. B. Rawlings, D. Q. Mayne, M. M. Diehl,
{\it Model Predictive Control: Theory, Computation, and Design}
Nob Hill Publishing, 2017.

\bibitem{maj}
J. M. Maciejowski
{\it Predictive Control with Constraints}
Addison-Wesley, 2002.


\bibitem{dahleh}
M. A. Dahleh, I. J. D{\'\i}az-Bobillo, 
{\it Control of Uncertain Systems: A Linear Programming Approach}  
Prentice Hall, Englewood Cliffs, NJ, 1995.

\bibitem{elia}
N. Elia, M. A. Dahleh, 
{\it Computational methods for controller design}  
Lecture Notes in Control and Information Sciences, Springer, 1998

\bibitem{zhou}
K. Zhou, J. C. Doyle,  K. Glover 
{\it Robust and Optimal Control}  
Prentice Hall, Upper Saddle River, NJ, 1996.

\bibitem{basar}
T. Basar, P. Bernhard,
{\it $H^{\infty}$-Optimal Control and Related Minimax Design Problems}   
Birkh{\"a}user, Boston, MA, 1995.

\bibitem{dullerud}
G. E. Dullerud, F. Paganini,
{\it A Course in Robust Control Theory: A Convex Approach}
Springer, NY, 2010.

\bibitem{bental1}
A. Ben-Tal, A. S. Nemirovski, L. El Ghaoui L, 
{\it Robust Optimization} 
Princeton University Press, Princeton, NJ, 2009.

\bibitem{krasovskii}
N. Krasovskii and E. Lidskii, 
Analytical design of controllers in systems with random attributes I,II,III, 
{\it Automat. Remote Control}, 22: 1021–1025, 1961.

\bibitem{sworder}
 D. Sworder, 
Feedback control of a class of linear systems with jump parameters, 
{\it IEEE Trans. Automat. Control}, 14(1), 9-–14, 1969.

\bibitem{wohnham}
W. M. Wonham, 
Random differential equations in control theory, 
{\it Probabilistic Methods in Applied Mathematics },  
New York, Academic Press, 2: 131–-212, 1970. 

\bibitem{costa}
O. L. V. Costa, M. D. Fragoso, R. P. Marques, 
{\it Discrete-Time Markov Jump Linear Systems }   
Springer-Verlag, New York, NY, 2005.

\bibitem{morozan}
V. Dragan, T. Morozan, A.M. Stoica
{\it Mathematical methods in robust control of discrete-time linear stochastic systems } 

\bibitem{seiler}
P. Seiler,  R. Sengupta,  
An $H_{\infty}$ approach to networked control.
{\it IEEE Trans. Automat. Control}, 50(3), 356–-364, 2005

\bibitem{hespanha}
J. P. Hespanha, P. Naghshtabrizi, Y. G. Xu
A Survey of Recent Results in Networked Control Systems.
{\it Proceedings of the IEEE}, 95(1), 138-162, 2007

\bibitem{kwon}
R. H. Kwon,  J. Y. Li,
A stochastic semidefinite programming approach for bounds on option pricing under regime switching.
{\it Ann. Oper. Res.} 237: 41--75, 2016.

\bibitem{tu}
J. Tu J,
Is regime switching in stock returns important in portfolio decisions?
{\it Management Sci.} 56(7): 1198--1215, 2010.

\bibitem{hamilton}
J.D. Hamilton, 
A new approach to the economic analysis of nonstationary time
series and the business cycle. 
{\it Econometrica} 57(2): 357--384, 1989.

\bibitem{jud_17}
A. Juditsky , A. S. Nemirovski,   
Near-optimality of linear recovery in Gaussian
observation scheme under $\| \cdot \|_2^2 $ loss.
{\it Ann. of Statist.} 46(4):1603-1629, 2018.

\bibitem{chen_1}
Y. Chen, T. T. Georgiou, M. Pavon,
Optimal Steering of a Linear Stochastic System to a Final Probability Distribution, Part I.
{\it IEEE Trans. on Automatic Control} 61(5):1158-1169, 2016

\bibitem{chen_2}
Y. Chen, T. T. Georgiou, M. Pavon,
Optimal Steering of a Linear Stochastic System to a Final Probability Distribution, Part II.
{\it IEEE Trans. on Automatic Control} 61(5):1170--1180, 2016

\bibitem{chen_3}
Y. Chen, T. T. Georgiou, M. Pavon,
Optimal Steering of a Linear Stochastic System to a Final Probability Distribution, Part III.
{\it IEEE Trans. on Automatic Control} 63(9):3112-3118, 2016

\bibitem{okamoto}
K. Okamoto, M. Goldshtein, P. Tsiotras, 
Optimal Covariance Control for Stochastic Systems Under Chance Constraints.
{\it IEEE Control Systems Letters} 2(2):266-271, 2018. 

\bibitem{bakolas}
E. Bakolas,
Finite-horizon covariance control for discrete-time stochastic linear
systems subject to input constraints.
{\it Automatica} 91:61--68, 2018.

\bibitem{hotz1}
A. F. Hotz, R. E. Skelton, 
A covariance control theory.
{\it 24th IEEE Conference on Decision and Control } Fort Lauderdale, FL, USA, 552--557, 1985.

\bibitem{hotz2}
A. F. Hotz, R. E. Skelton, 
Covariance control theory.
{\it International Journal of Control } 46(1) 13-32, 1987.

\bibitem{xu}
J. Xu, R. E. Skelton, 
An improved covariance assignment theory for discrete systems, 
{\it IEEE Trans. on Automatic Control} 37(10):1588--1591, 1992.

\bibitem{yasuda}
K. Yasuda, R. E. Skelton, K. M. Grigoriadis,
Covariance controllers: A new parametrization of the class of all stabilizing controllers,
{\it Automatica} 29(3):785--788, 1993.

\bibitem{grigoriadis}
K. M. Grigoriadis, R. E. Skelton,
Minimum-energy covariance controllers.
{\it Automatica} 33(4):785--788, 1997.

\bibitem{skelton}
R.E. Skelton, T. Iwasaki, K. Grigoriadis
{\it A Unified Algebraic Approach to Linear Control Design}  
Taylor \& Francis, London, UK, 1998


\bibitem{bental3}
A. Ben-Tal , A. Goryashko, E. Guslitzer, A.S Nemirovski,
Adjustable robust solutions of uncertain linear programs.
{\it Math. Programming Ser. A} 99: 351--376, 2004.

\bibitem{bental4}
A. Ben-Tal, A.S. Nemirovski,  
Robust convex optimization.
{\it Math. of  Oper.  Res.} 23(4): 769--805, 1998.

\bibitem{bental5}
A. Ben-Tal, A.S. Nemirovski,
Selected topics in robust convex optimization.
{\it Math. Programming Ser. B} 112: 125--158, 2008.

\bibitem{bertsimas1}
D. Bertsimas, D. B. Brown, C. Caramanis,  
Theory and applications of robust optimization. 
{\it SIAM Rev.} 53(3):464--501, 2011.

\bibitem{sim1}
D. Bertsimas, M. Sim,   
Tractable approximations to robust conic optimization
Problems.
{\it Math. Programming Ser. B} 107: 5--36, 2006.

\bibitem{brown1}
D. Bertsimas, D. B. Brown,
Constrained stochastic LQC: A tractable approach
{\it IEEE Trans. Automatic Control} 52(10):464--501, 2007.

\bibitem{bental2}
A. Ben-Tal, S. Boyd S, A.S. Nemirovski, 
Extending scope of robust optimization: Comprehensive robust counterparts of uncertain problems.
{\it Math. Programming Ser. B} 107: 63--89, 2006

\bibitem{youla1}
D. Youla, J. Bongiorno and H. Jabr, 
Modern Wiener--Hopf design of optimal controllers Part I: The single-input-output case, 
{\it IEEE Trans. on Automatic Control} 21(1):3-13, 1976

\bibitem{youla2}
D. Youla, J. Bongiorno and H. Jabr, 
Modern Wiener--Hopf design of optimal controllers Part II: The multivariable case, 
{\it IEEE Trans. on Automatic Control} 21(3):319-338, 1976

\bibitem{barratt}
S. Boyd, C. Barratt,
{\it Linear Controller Design – Limits of Performance}  
Prentice-Hall, 1991

\bibitem{vid}
M. Vidyasagar, 
{\it Control System Synthesis. A Factorization Approach} 
The MIT Press, Cambridge, MA, 1985.


\bibitem{bertsimas2}
D. Bertsimas, D. A. Iancu, P.A. Parrilo, 
Optimality of affine policies in multistage robust optimization.
{\it Math. of  Oper.  Res.} 35(2): 363--394, 2010.

\bibitem{bertsimas3}
D. Bertsimas, D. A. Iancu, P.A. Parrilo, 
A hierarchy of near-optimal policies for multistage adaptive optimization.
{\it IEEE Trans. Automatic Control} 56(12):2809--2823, 2009.

\bibitem{kuhn}
D. Kuhn , W. Wiesemann, A. Georghiou,
Primal and dual linear decision rules in stochastic and robust optimization.
{\it Math. Programming}  130(1): 177--209, 2011.

\bibitem{kwakernaak} 
H. Kwakernaak, R. Sivan  (1972)
{\it Linear Optimal Control Systems} 
Wiley, New York, NY, 1972.

\bibitem{andreson}
B.D.O. Anderson, J. B. Moore,   
{\it  Optimal Control - Linear Quadratic Methods} 
Prentice Hall, Englewood Cliffs, NJ, 1990.

\bibitem{whittle2}
P. Whittle, 
{\it Risk-Sensitive Optimal Control} 
Wiley, New York, NY, 1990.

\bibitem{sun}
X. Chen, M.  Sim, P. Sun,J. Zhang,
A linear decision-based approximation approach to stochastic programming.
{\it Oper. Res.} 56(2):344--357, 2008.

\bibitem{nemirovski}
A.S. Nemirovski, A. Shapiro 
On complexity of stochastic programming problems.
{\it Continuous Optimization: Current Trends and Applications}  
Springer-Verlag, New York, NY  111--146, 2005


\bibitem{loefberg}
J. L{\"o}fberg 
Approximations of closed-loop MPC. 
{\it Proceedings of the 42nd IEEE Conf. on Decision and Control} 1438--1442, 2003.

\bibitem{goulart1}
P. J. Goulart, E. C. Kerrigan,J. M.  Maciejowski,
Optimization over state feedback policies for robust control with constraints.
{\it Automatica} 42: 523--533, 2006.

\bibitem{goulart2}
P. J. Goulart, E. C. Kerrigan,
Output feedback receding horizon control of constrained systems.
{\it International Journal of Control} 80: 8--20, 2007.

\bibitem{skaf}
J. Skaf, S. Boyd, 
Design of affine controllers via convex optimization.
{\it IEEE Trans. on Automatic Control} 55(11): 2476--2487, 2010.

\bibitem{charnes}
A. Charnes, W.W. Cooper, G. H. Symonds,
Cost horizons and certainty equivalents: An approach to stochastic programming of heating oil.
{\it Management Sci.} 4(3): 235--263, 1954.

\bibitem{gartska}
S. J. Gartska , R. J. B. Wets,
On decision rules in stochastic programming.
{\it Math. Programming} 7(1): 117-143, 1974.

\bibitem{salm}
D. Salmon,  
Minimax controller design. 
{\it IEEE Trans. on Automatic Control} 13(4):369--376, 1968.

\bibitem{wits}
H. S. Witsenhausen, 
A minimax control problem for sampled linear systems. 
{\it IEEE Trans. on Automatic Control} 13(1):5--21, 1968.

\bibitem{glover}
J. D. Glover, F.C. Schweppe, 
Control of linear dynamic systems with set constrained disturbances.
{\it IEEE Trans. on Automatic Control} 16: 411 -- 423, 1971.

\bibitem{ber2}
D.P. Bertsekas, I.B. Rhodes,  
On the minmax reachability of target set and target tubes.
{\it Automatica} 7: 233-247, 1971.

\bibitem{ber3}
D.P. Bertsekas, I.B. Rhodes,  
Recursive state estimation for a set-membership description of uncertainty.
{\it IEEE Trans. on Automatic Control} 16: 117-128, 1971.

\bibitem{ber4}
D. P. Bertsekas, 
Infinite-time reachability of state-space regions by using feedback control.
{\it IEEE Trans. on Automatic Control} 17(5): 604--613, 1972.

\bibitem{yakubovich1}
V. A. Yakubovich, 
S-procedure in nonlinear control theory.
{\it Vestnik Leningrad. Univ. } 4:73-93, 1977, (English translation).

\bibitem{meg} 
A. Megretski, S. Treil, 
Power distribution in optimization and robustness of uncertain systems.
{\it J. Math. Systems Estim. Control }  3: 301-319, 1993. 

\bibitem{yakubovich2}
V. A. Yakubovich, 
Nonconvex optimization problem: The
infinite-horizon linear-quadratic control
problem with quadratic constraints. 
{\it Systems \& Control Letters } 19: 13-22, 1992

\bibitem{fer} 
S. Boyd, L. El Ghaoui, E. Feron,  V. Balakrishnan
{ \it Linear Matrix Inequalities in System and Control Theory} 
SIAM, Philadelphia, PA, 1994.

\bibitem{savkin} 
I. Petersen,
V. A. Ugrinovskii,
A. V. Savkin,
{ \it Robust Control Design Using $H_{\infty}$  Methods} 
Springer-Verlag London 2000.

\bibitem{bental6}
A. Ben-Tal , A. S. Nemirovski,    
{\it Lectures on Modern Convex Optimization: Analysis, Algorithms, Engineering Applications }   
SIAM, Philadelphia, PA, 2001.


\bibitem{cvx}
M. Grant, S. Boyd,
CVX: MATLAB Software for Disciplined Convex Programming, Version 2.1, 2014  

\bibitem{GrantBoydYe:2006}
M. Grant,S. Boyd,Y. Ye, 
Disciplined Convex Programming.
In L~Liberti, N~Maculan (eds.), \emph{Global Optimization: From
	Theory to Implementation}, Nonconvex Optimization and its Applications, pp.
155--210. Springer, 2006.


\bibitem{sdpt3}
K. C. Toh, M. J. Todd, R. H. Tutuncu 
SDPT3 --- a Matlab software package for semidefinite programming, {\it Optimization Methods and Software}, 11: 545--581, 1999.

 


\end{thebibliography}
\end{document}